\documentclass[11pt,reqno]{amsart}
\usepackage[T1]{fontenc}
\usepackage{amssymb,amsmath}
\usepackage[english]{babel}
\usepackage{enumerate}
\usepackage{enumitem,kantlipsum}
\usepackage[colorlinks=true,linkcolor=blue]{hyperref}
\usepackage{dsfont}
\usepackage{cite}
\usepackage[foot]{amsaddr}
\usepackage{nicefrac}

\newtheorem{theorem}{Theorem}[section]
\newtheorem{lemma}[theorem]{Lemma}
\newtheorem{corollary}[theorem]{Corollary}
\theoremstyle{definition}
\newtheorem{definition}[theorem]{Definition}
\newtheorem{example}[theorem]{Example}
\newtheorem{remark}[theorem]{Remark}
\newtheorem{assumption}[theorem]{Assumption}

\setlist[enumerate]{leftmargin=*, label=(\roman*)}
\numberwithin{equation}{section}

\allowdisplaybreaks

\def\Ci{{\rm C}^\infty}
\def\Ck{{\rm C}_\kappa}
\def\Cb{{\rm C}_{\rm b}}
\def\Cbi{{\rm C}_{\rm b}^\infty}
\def\Lip{{\rm Lip}}
\def\Lipb{{\rm Lip}_{\rm b}}

\DeclareMathOperator{\ca}{ca}

\DeclareMathOperator*{\Glimsup}{\Gamma-\limsup}
\DeclareMathOperator{\id}{id}
\DeclareMathOperator{\supp}{supp}
\DeclareMathOperator{\tr}{tr}

\def\A{\mathcal{A}}
\def\B{\mathcal{B}}
\def\d{{\rm d}}
\def\E{\mathcal{E}}
\def\e{\mathbb{E}}
\def\H{\mathcal{H}}
\def\L{\mathcal{L}}
\def\N{\mathbb{N}}
\def\R{\mathbb{R}}
\def\S{\mathbb{S}}

\def\epsilon{\varepsilon}

\def\AG{A_\Gamma}

\newcommand{\Newline}{{\rule{0mm}{1mm}\\[-3.25ex]\rule{0mm}{1mm}}}

\begin{document}

\title[Convergence rates Chernoff]{Convergence rates for Chernoff-type approximations of convex monotone semigroups}

\author{Jonas Blessing$^1$}
\address{{\rm $^1$ETH Zurich, Department of Mathematics, R\"amistra\ss{}e 101, 8092 Zurich, Switzerland}}
\email{$^1$jonas.blessing@math.ethz.ch}

\author{Lianzi Jiang$^2$}
\address{{\rm $^2$Shandong University of Science and Technology, College of Mathematics and Systems Science, 
Qingdao, Shandong 266590, China}}
\email{$^2$jianglianzi95@163.com}

\author{Michael Kupper$^3$}
\address{{\rm $^3$University of Konstanz, Department of Mathematics and Statistics, 78457 Konstanz, Germany}}
\email{$^3$kupper@uni-konstanz.de}

\author{Gechun Liang$^4$}
\address{{\rm $^4$The University of Warwick, Department of Statistics, Coventry CV4 7AL, United Kingdom}}
\email{$^4$g.liang@warwick.ac.uk}

\date{\today}

\thanks{Financial support through the National Natural Science Foundation of Shandong Province 
(No. ZR2023QA090).}

\begin{abstract}
 We provide explicit convergence rates for Chernoff-type approximations of convex 
 monotone semigroups which have the form $S(t)f=\lim_{n\to\infty}I(\nicefrac{t}{n})^n f$ 
 for bounded continuous functions $f$. Under suitable conditions on the one-step 
 operators $I(t)$ regarding the time regularity and consistency of the 
 approximation scheme, we obtain 
 $\|S(t)f-I(\nicefrac{t}{n})^n f\|_\infty\leq cn^{-\gamma}$ 
 for bounded Lipschitz continuous functions $f$, where $c\geq 0$ and $\gamma>0$ 
 are determined explicitly. Moreover, the mapping $t\mapsto S(t)f$ is H\"older continuous.\
 These results are closely related to monotone approximation schemes for viscosity 
 solutions but are obtained independently by following a recently developed semigroup 
 approach to Hamilton--Jacobi--Bellman equations which uniquely characterizes semigroups via their 
 $\Gamma$-generators.\ The different approach allows to consider convex rather 
 than sublinear equations and the results can be extended to unbounded functions
 by modifying the norm with a suitable weight function.\ Furthermore, up to possibly 
 different consistency errors for the operators $I(t)$, the upper and lower bound 
 for the error between the semigroup and the iterated operators are symmetric.
 The abstract results are applied to Nisio semigroups and limit theorems for convex 
 expectations. 
\end{abstract}

\maketitle

\section{Introduction}

In this article, we provide explicit convergence rates for Chernoff-type approximation
schemes of strongly continuous convex monotone semigroups which are closely
related to monotone approximation schemes for viscosity solutions. Initially, it was shown 
by Barles and Souganidis~\cite{BS1991} that viscosity solutions to second-order parabolic 
PDEs, which satisfy a comparison principle, can be approximated by any consistent and 
stable monotone scheme. The question of the corresponding convergence rates for 
Hamilton--Jacobi--Bellman (HJB) equations related to stochastic optimal 
control problems was then addressed in a series of articles by Krylov,
see~\cite{Krylov1997,Krylov99, Krylov2000}. The novelty in his approach is 
the so-called shaking coefficients technique which allows to construct an approximating 
sequence of smooth subsolutions yielding a one-sided error bound. Since these arguments 
rely on the convexity of the underlying equation, obtaining the other bound is much more 
challenging. The first idea is to interchange the roles of the solution and the approximation 
scheme in which case one has to show that solutions to equations with shaken coefficients
remain close to the solution of the original equation, see Barles and Jakobsen~\cite{BJ2002},
Dong and Krylov~\cite{DK05,DK05b}, Jakobsen~\cite{Jakobsen2003} and Krylov~\cite{Krylov05}. 
The second idea combines the shaking coefficients technique with an approximation by optimal 
switching systems which applies in a more general framework but possibly at the cost of lower 
rates, see Barles and Jacobsen~\cite{BJ2005, BJ07}. These results have been used to derive 
explicit convergence rates for several numerical schemes for HJB equations, see, e.g., Bayraktar and 
Fahim~\cite{BF2014}, Bokanowski et al.~\cite{BMZ2009}, Briani et al.~\cite{BCZ12}, Caffarelli 
and Souganidis~\cite{CS08}, Fahim et al.~\cite{FTW2011}, Huang et al.~\cite{HLZ2020} and 
Jiang~\cite{Jiang2022}. Regularity of optimal switching systems has been further investigated 
by Picarelli et al.~\cite{PRR2020} and, by using a higher order Taylor expansion, previous 
convergence rates have been improved by Jakobsen et al.~\cite{JPR2019}. These methods
also apply in the context of limit theorems for sublinear expectations which have been introduced
by Peng~\cite{Peng07, Peng08, Peng08b} and then underwent many extensions,
see, e.g., Hu et.\ al~\cite{HJLP2023}. Indeed, the convergence rates obtained by 
Krylov~\cite{Krylov20}, Hu at. al~\cite{HJL21} and Huang and Liang~\cite{HL19} improve previous 
results by Fang et. al~\cite{FPSS2019} and Song~\cite{Song2020} based on Stein's method.

In contrast to the previously mentioned results, the present work does not rely on the theory 
of viscosity solutions but follows a semigroup approach to HJB equations developed in a 
series of papers by Blessing and Kupper~\cite{BK22,BK22+}, Blessing et.\ al~\cite{BDKN22, BKN23},
Denk et.\ al~\cite{DKN20} and Nendel and R\"ockner~\cite{NR21}.\ 
Key results are a comparison principle, which uniquely determines strongly continuous 
convex monotone semigroups by their $\Gamma$-generators, and the construction
of semigroups via Chernoff-type approximation schemes. The latter are described by 
operator families $(I(t))_{t\geq 0}$ which do not form a semigroup but have the desired
infinitesimal behaviour
\[ I'(0)f:=\lim_{h\downarrow 0}\frac{I(h)f-f}{h} \]
for sufficiently smooth functions $f$. The corresponding semigroup is then given by
\[ S(t)f:=\lim_{n\to\infty}I\big(\tfrac{t}{n}\big)^n f 
	=\lim_{n\to\infty}\underbrace{\big(I\big(\tfrac{t}{n}\big)\circ\ldots\circ I\big(\tfrac{t}{n}\big)\big)}_{n \mbox{ \small times}}f \]
and is uniquely determined via its infinitesimal generator $Af=I'(0)f$.
The contribution of this article is to provide convergence rates of the form
\begin{equation} \label{eq:intro}
 \|S(t)f-I\big(\tfrac{t}{n}\big)^n f\|_\infty\leq cn^{-\gamma}
\end{equation}
for bounded Lipschitz continuous functions $f$, where $c\geq 0$ and $\gamma>0$
are determined explicitly. To do so, we suppose that the time regularity and consistency 
of the approximation scheme can be controlled by some nonlinear functions $\rho_1$,
$\rho_2$ and $\rho_3$ depending on the partial derivatives up to order two of smooth 
test functions. The resulting Theorem~\ref{thm:rate} resembles in its generality the work 
in~\cite{BJ07} but we are not restricted to the sublinear case and the proof does neither 
involve shaken coefficients nor approximations by switching systems. Hence, up to 
possibly different consistency errors, the lower and upper bound are symmetric. 
Furthermore, the time regularity of $(I(t))_{t\geq 0}$ transfers to $(S(t))_{t\geq 0}$,
see Theorem~\ref{thm:time}. If the functions $\rho_1$, $\rho_2$ and $\rho_3$ are polynomials, 
optimizing over the parameters in Theorem~\ref{thm:rate} and Theorem~\ref{thm:time} 
yields convergence rates in the form of inequality~\eqref{eq:intro} and H\"older regularity 
in time, see Theorem~\ref{thm:rate2} and Corollary~\ref{cor:time}. Finally, the results
can be extended to arbitrary Lipschitz functions at the cost of weakening the norm
with a weight function, see Theorem~\ref{thm:rate3}.

The abstract results are first applied to Nisio semigroups which are generated by a family 
$(S_\lambda)_{\lambda\in\Lambda}$ of linear semigroups $(S_\lambda(t))_{t\geq 0}$, i.e.,
we choose $I(t)f:=\sup_{\lambda\in\Lambda}S_\lambda(t)f$.
The generator of the corresponding semigroup $(S(t))_{t\geq 0}$ is given by
$Af=\sup_{\lambda\in\Lambda}A_\lambda$, where $A_\lambda$ denotes the generator
of $(S_\lambda(t))_{t\geq 0}$. This approach to nonlinear semigroups, which often represent
the value functions of dynamic stochastic control problems, dates back to 
Nisio~\cite{Nisio76} and has been further developed by Denk et.\ al~\cite{DKN20} 
and Nendel and R\"ockner~\cite{NR21}. Under suitable conditions on the family 
$(A_\lambda)_{\lambda\in\Lambda}$, the convergence rate is $\nicefrac{1}{2}$ for first 
order operators and $\nicefrac{1}{6}$ in the second order case, see Theorem~\ref{thm:nisio}. 
Additional regularity even yields the rate $\nicefrac{1}{4}$ in the second order case, 
see Theorem~\ref{thm:nisio2}. We then focus on the law of large numbers (LLN) and central 
limit theorem (CLT) type results for convex expectations which have been obtained in~\cite{BK22}. 
While the rate for the LLN is always $\nicefrac{1}{2}$ the rate for the CLT depends on the growth 
behaviour of the convex expectation, see Theorem~\ref{thm:LLN} and Theorem~\ref{thm:CLT}.\ 
In one dimension, the rate for the CLT can be improved by imposing an additional moment 
condition, see Theorem~\ref{thm:CLT2}. Furthermore, in the sublinear case, the rate for the 
CLT becomes $\nicefrac{1}{6}$ respectively $\nicefrac{1}{4}$ which is consistent with the 
rates for Nisio semigroups and previous results, see~\cite{Krylov20,HL19}.

The paper is organised as follows. In Section~\ref{sec:main}, we state the main results
regarding the convergence rates and the time regularity of the semigroup.\
The corresponding proofs are given in Section~\ref{sec:proof} and the examples
are presented in Section~\ref{sec:examples}.

\section{Setup and main results}
\label{sec:main}

Let $\kappa\colon\R^d\to (0,1]$ be a continuous function with
\begin{equation} \label{eq:kappa}
 c_\kappa:=\sup_{x\in\R^d}\sup_{|y|\leq 1}\frac{\kappa(x)}{\kappa(x-y)}<\infty
\end{equation} 
and denote by $\Ck$ the space of of all continuous functions $f\colon\R^d\to\R$ 
satisfying
\[ \|f\|_\kappa=\sup_{x\in\R^d}|f(x)|\kappa(x)<\infty. \]
The space $\Ck$ is endowed with the mixed topology between $\|\cdot\|_\kappa$ and the topology 
of uniform convergence on compact sets which is the strongest locally convex topology on $\Ck$ 
that coincides on $\|\cdot\|_\kappa$-bounded sets with the topology of uniform convergence on
compact sets, see~\cite{GNR22, FGH72,Sentilles72}.\ Although the mixed topology is not metrizable, 
it has recently been shown in~\cite{Nendel22} that, for monotone operators $S\colon \Ck\to \Ck$, 
sequential continuity is equivalent to continuity. Furthermore, a sequence $(f_n)_{n\in\N}\subset\Ck$ 
converges to $f\in\Ck$ w.r.t. the mixed topology if and only if
\[ \sup_{n\in\N}\|f_n\|_\kappa<\infty \quad\mbox{and}\quad
\lim_{n\to\infty}\|f-f_n\|_{\infty,K}=0 \]
for all compact subsets $K\Subset\R^d$ and $\|f\|_{\infty,K}:=\sup_{x\in K}|f(x)|$, see~\cite{GNR22}.
Subsequently, if not stated otherwise, all limits in $\Ck$ are understood w.r.t. the mixed topology
and inequalities between functions apply pointwise. The space $\Lipb$ consists of all bounded 
Lipschitz continuous functions $f\colon\R^d\to\R$ and, for every $r\geq 0$, the set $\Lipb(r)$ contains 
all $r$-Lipschitz functions $f\colon\R^d\to\R$ with $\|f\|_\infty:=\sup_{x\in\R^d}|f(x)|\leq r$.
In addition, the space $\Cbi$ consists of all bounded infinitely differentiable functions
$f\colon\R^d\to\R$ such that the partial derivatives of any order are bounded.
For every $f\in\Cbi$ and $l\in\N$, let
\[ D^l f:=(D^\gamma f)_{|\gamma|=l} \quad\mbox{and}\quad 
	\|D^l f\|_\kappa:=\left\|\left(\sum\nolimits_{|\gamma|=l}|D^\gamma f|^2\right)^{1/2}\right\|_\kappa \]
be the family of all partial derivatives $D^\gamma f:=\partial_{x_1}^{\gamma_1}\cdots\partial_{x_d}^{\gamma_d}f$
of order $|\gamma|:=\sum_{i=1}^d \gamma_i=l$ and its norm.
We define $D^0f:=f$ and identify $Df:=D^1f$ with the gradient and $D^2f$ with the Hessian.
In particular, by endowing $\R^d$ with the Euclidean norm, it holds
\[ \|Df(\cdot)^T\xi_1\|_\kappa\leq\|Df\|_\kappa |\xi_1| \quad\mbox{and}\quad
	\|\xi_1^T D^2 f(\cdot) \xi_2\|_\kappa\leq\|D^2f\|_\kappa |\xi_1| |\xi_2| \]
for all $\xi_1,\xi_2\in\R^d$. Let $\R_+:=\{x\in\R\colon x\geq 0\}$ and denote closed balls by
\[ B_{\Ck}(r):=\{f\in\Ck\colon\|f\|_\kappa\leq r\} \quad\mbox{and}\quad
	B_{\R^d}(r):=\{x\in\R^d\colon |x|\leq r\}. \]
For the rest of this section, let $\eta\colon\R_+\times\R^d\to\R_+$ be a fixed infinitely differentiable function
with $\supp(\eta)\subset [0,1]\times B_{\R^d}(1)$ and $\int_{\R_+\times\R^d}\eta(s,y)\,\d s\,\d y=1$.
For every locally bounded measurable function $u\colon\R_+\times\R^d\to\R$, $\epsilon=(\epsilon_1,\epsilon_2)$ 
with $\epsilon_1,\epsilon_2>0$, $t\geq 0$ and $x\in\R^d$, we define
\begin{equation} \label{eq:eta}
	u^\epsilon(t,x):=(u*\eta^\epsilon)(t,x)
	:=\int_{\R_+\times\R^d}u(s+t, x+y)\eta^{\epsilon}(s,y)\,\d s\,\d y,
\end{equation}
where $\eta^{\epsilon}(s,y):=\epsilon_1^{-1}\epsilon_2^{-d}\eta(\epsilon_1^{-1}s,\epsilon_2^{-1}y)$.
For every $k,l\in\N_0$, the constants
\begin{equation}\label{const:b}
 b_{k,l}:=\max_{|\alpha|=l}\big\|\partial_t^kD^\alpha\eta\big\|_{L^1(\R_+\times\R^d)}
\end{equation}	
depend on the dimension $d$ and the choice of $\eta$.

\subsection{Chernoff-type approximations in the mixed topology}

\begin{definition}	
 Let $(I(t))_{t\geq 0}$ be a family of operators $I(t)\colon\Ck\to\Ck$ with $I(0)=\id_{\Ck}$.
 The Lipschitz set $\L^I$ consists of all $f\in\Ck$ such that there exist $c\geq 0$ and
 $t_0>0$ with
 \[ \|I(t)f-f\|_\kappa\leq ct \quad\mbox{for all } t\in [0,t_0]. \]
 Moreover, for every $f\in\Ck$ such that the following limit exists, we define the derivative of the mapping
 $t\mapsto I(t)f$ at zero by
 \[ I'(0)f:=\lim_{h\downarrow 0}\frac{I(h)f-f}{h}\in\Ck. \]
\end{definition}

\begin{definition}
 A family $(S(t))_{t\geq 0}$ of operators $S(t)\colon\Ck\to\Ck$ is called semigroup if
 $S(0)=\id_{\Ck}$ and $S(s+t)f=S(s)S(t)f$ for all $s,t\geq 0$ and $f\in\Ck$. The semigroup
 is convex (monotone) if the operators $S(t)\colon\Ck\to\Ck$ are convex (monotone)
 for all $t\geq 0$. Moreover, the semigroup is strongly continuous if the mapping
 $\R_+\to\Ck,\; t\mapsto S(t)f$ is continuous for all $f\in\Ck$. The generator of the semigroup
 is defined by
 \[ A\colon D(A)\to\Ck,\; f\mapsto\lim_{h\downarrow 0}\frac{S(h)f-f}{h}, \]
 where the domain $D(A)$ consists of all $f\in\Ck$ such that the previous limit exists.
\end{definition}

Let $(I(t))_{t\geq 0}$ be a family of operators $I(t)\colon\Ck\to\Ck$ and let $(h_n)_{n\in\N}\subset (0,1]$
be a sequence with $h_n\to 0$. For every $t\geq 0$ and $n\in\N$, we define the iterated operators
\begin{equation} \label{eq:Ipi}
 I(\pi_n^t)\colon\Ck\to\Ck, \; f\mapsto I(h_n)^{k_n^t}f,
\end{equation}
where $k_n^t:=\max\{k\in\N_0\colon kh_n\leq t\}$ and $\pi_n^t:=\{0, h_n, \ldots, k_n^t h_n\}$
is the corresponding equidistant partition with mesh size $h_n$. 
By definition, it holds $I(\pi_n^t)f=I(h_n)^0 f:=f$ for all $t\in [0, h_n)$. 
In previous works on Chernoff-type approximations such as~\cite{BK22+, BK22, BDKN22}, 
the iterated operators are defined by $ I(\pi^t_n):=I(h_n)^{k_n^t} I(t-k_n^th_n)$. 
Defining the iterated operators by equation~\eqref{eq:Ipi}
instead does not affect the results in~\cite{BK22+, BK22, BDKN22} and is closer to the literature on 
monotone schemes. Furthermore, due to the comparison principle for strongly continuous convex
monotone semigroups in~\cite{BDKN22, BK22, BKN23}, both definitions of $I(\pi_n^t)$ lead to 
the same corresponding semigroup given by equation~\eqref{eq:cher} below. 
We define $(\tau_x f)(y):=f(x+y)$ for all $f\colon\R^d\to\R$ and $x,y\in\R^d$.

\begin{assumption} \label{ass:cher}
 Suppose that the following statements are valid:
 \begin{enumerate}
  \item $I(0)=\id_{\Ck}$.
  \item $I(t)$ is convex and monotone with $I(t)0=0$ for all $t\geq 0$.
  \item There exists $\omega\geq 0$ with
   \[ \|I(t)f-I(t)g\|_\kappa\leq e^{\omega t}\|f-g\|_\kappa
   	\quad\mbox{for all } t\in [0,1] \mbox{ and } f,g\in\Ck. \]
  \item There exist $t_0>0$, $\epsilon_0\in (0,1]$ and $L\geq 0$ with
   \[ \|I(t)(\tau_x f)-\tau_x I(t)f\|_\kappa\leq Lrt|x| \]
   for all $t\in [0,t_0]$, $x\in B_{\R^d}(\epsilon_0)$, $r\geq 0$ and $f\in\Lipb(r)$.
  \item It holds $\Cbi\subset\L^I$ and $I'(0)f\in\Ck$ exists for all $f\in\Cbi$.
  \item For every $t\geq 0$, $K\Subset\R^d$ and $(f_k)_{k\in\N}\subset\Ck$ with
   $f_k\downarrow 0$,
   \[ \sup_{(s,x)\in [0,t]\times K}\sup_{n\in\N}\big(I(\pi_n^s)f_k\big)(x)\downarrow 0
   	\quad\mbox{as } k\to \infty. \]
   \item It holds $I(t)\colon\Lipb(r)\to\Lipb(e^{\omega t}r)$ for all $r,t\geq 0$.
 \end{enumerate}
\end{assumption}

Note that, apart from condition~(vi), we only impose conditions on the one-step operators
$(I(t))_{t\geq 0}$ which are preserved during the iteration and thus transfer to the semigroup
in the limit. Furthermore, in~\cite[Subsection~2.5]{BKN23}, the authors provide sufficient
conditions on the one-step operators $(I(t))_{t\geq 0}$ in order to guarantee that condition~(vi)
is valid for the iterated operators. For more details, we refer to Section~\ref{sec:examples},
where the abstract results are illustrated with several examples.

\begin{theorem} \label{thm:cher}
 Let $(I(t))_{t\geq 0}$ be a family of operators satisfying Assumption~\ref{ass:cher}
 with constants $\omega\geq 0$, $\epsilon_0\in (0,1]$ and $L\geq 0$.
 Then, there exists a strongly continuous convex monotone semigroup $(S(t))_{t\geq 0}$
 on $\Ck$ with $S(t)0=0$ and
 \begin{equation} \label{eq:cher}
  S(t)f=\lim_{n\to\infty}I(\pi_n^t)f \quad\mbox{for all } t\geq 0 \mbox{ and } f\in\Ck.
 \end{equation}
 Furthermore, the semigroup has the following properties:
 \begin{enumerate}
  \item It holds $f\in D(A)$ and $Af=I'(0)f$ for all $f\in\Ck$ such that $I'(0)f\in\Ck$ exists.
   In particular, this is valid for all $f\in\Cbi$.
  \item It holds $\|S(t)f-S(t)g\|_\kappa\leq e^{\omega t}\|f-g\|_\kappa$ for all $t\geq 0$
   and $f,g\in\Ck$.  
  \item For every $\epsilon>0$, $r,t\geq 0$ and $K\Subset\R^d$, there exist $K'\Subset\R^d$
   and $c\geq 0$ with
   \[ \|S(s)f-S(s)g\|_{\infty,K}\leq c\|f-g\|_{\infty,K'}+\epsilon \]
   for all $s\in [0,t]$ and $f,g\in B_{\Ck}(r)$.
  \item It holds $\L^I\subset\L^S$ and $S(t)\colon\L^S\to\L^S$ for all $t\geq 0$.
  \item  For every  $r, t\geq 0$, $f\in\Lipb(r)$ and $x\in B_{\R^d}(\epsilon_0)$,
    \[ \|S(t)(\tau_x f)-\tau_x S(t)f\|_\kappa\leq Lrte^{2\omega t}|x|. \]
  \item It holds $S(t)\colon\Lipb(r)\to\Lipb(e^{\omega t}r)$ for all $r,t\geq 0$.
 \end{enumerate}
\end{theorem}

For a proof, we refer to~\cite[Theorem~2.8]{BK22} which is built on the results in~\cite[Section~4]{BDKN22}.
In addition, due to~\cite[Theorem~2.4]{BKN23}, which is an immediate consequence 
of the results in~\cite{BDKN22, BK22}, the semigroup $(S(t))_{t\geq 0}$ does not depend
on the choice of the sequence $(h_n)_{n\in\N}$ but only on the time derivative $I'(0)f$ for $f\in\Cbi$.

\subsection{Convergence rates}

The main result of this article is to provide convergence rates for the limit in equation~\eqref{eq:cher}
under the following consistency condition.

\begin{assumption} \label{ass:rate} \Newline
 \begin{enumerate}
  \item There exists a function $\rho_1\colon\R_+^2\to\R_+$ with
   \[ \|I(h_n)f-f\|_\kappa\leq\rho_1\big(d^{-\frac{1}{2}}\|D f\|_\infty, d^{-1} \|D^2 f\|_\infty\big)h_n
   	\quad\mbox{for all } n\in\N \mbox{ and } f\in\Cbi.\]
   The function $\rho_1$ is non-decreasing in both arguments.
  \item There exists a function $\rho_2\colon\R_+^5\to\R_+$ with
   \[ \partial_t u^{\epsilon}(t)-Au^{\epsilon}(t)
   	-\frac{u^{\epsilon}(t)-I(h_n)u^{\epsilon}(t-h_n)}{h_n}
    \leq\frac{\rho_2(\epsilon_1, \epsilon_2, h_n, r,t)}{\kappa} \]
   for all $\varepsilon=(\epsilon_1,\epsilon_2)$ with $\epsilon_1, \epsilon_2\in (0,\epsilon_0]$, 
   $n\in\N$, $r\geq 0$, $t\geq h_n$ and $f\in\Lipb(r)$, where $u(t):=S(t)f$ and 
   $u^\epsilon:=u*\eta^\epsilon$.
  \item There exists a function $\rho_3\colon\R_+^5\to\R_+$ with
   \[ \partial_t u_n^{\epsilon}(t)-Au_n^{\epsilon}(t)
   	-\frac{u_n^{\epsilon}(t)-I(h_n)u_n^{\epsilon}(t-h_n)}{h_n}
    \geq -\frac{\rho_3(\epsilon_1, \epsilon_2, h_n, r,t)}{\kappa} \]
   for all $\epsilon=(\epsilon_1,\epsilon_2)$ with $\epsilon_1, \epsilon_2\in (0,\epsilon_0]$,
   $n\in\N$, $r\geq 0$, $t\geq h_n$ and $f\in\Lipb(r)$, where $u_n(t):=I(\pi_n^t)f$ and 
   $u_n^\epsilon:=u_n*\eta^\epsilon$.\ Moreover, the function $\rho_3$ is non-decreasing 
   in the last argument.
 \end{enumerate}
\end{assumption}

The scaling with $d^{-\frac{1}{2}}$ and $d^{-1}$ comes from the fact that we use 
the Euclidean norm, see Lemma~\ref{lem:deriv+} below.
In Subsection~\ref{sec:time}, we discuss how the previous assumptions can typically
be verified. We briefly comment on the roles of the variables $(\epsilon_1,\epsilon_2,h_n,r,t)$.
The variables $(r,t)$ are parameters which do not affect the exponent of the convergence rate.
In typical situations, the functions $\rho_2$ and $\rho_3$ decrease to zero as $h_n\to 0$. However,
since the function $u$ is usually only H\"older continuous in time and Lipschitz continuous
in space, the terms $\partial_t u^{\epsilon}(t)$ and $Au^{\epsilon}(t)$
explode for $\epsilon_1, \epsilon_2\to 0$ and so does the function~$\rho_2$.
The same considerations are valid for $u_n$ and $\rho_3$. Thus, in order to optimize
the convergence rate, we must choose the variables $(\epsilon_1, \epsilon_2)$ depending
on $h_n$. The following theorem is the main result of this article.

\begin{theorem} \label{thm:rate}
 Let $(I(t))_{t\geq 0}$ be a family of operators satisfying Assumption~\ref{ass:cher} and
 Assumption~\ref{ass:rate}(i). Denote by $(S(t))_{t\geq 0}$ the corresponding semigroup
 from Theorem~\ref{thm:cher}.
 \begin{enumerate}
  \item Suppose that Assumption~\ref{ass:rate}(ii) is valid. Then,
   \begin{align*}
    \|(S(t)f-I(\pi_n^t)f)^-\|_\kappa
    &\leq e^{\omega (t+h_n)}\big(2r\delta_1+\rho_1\big(r, b_{0,1}r\delta_1^{-1}\big)h_n\big) \\
    &\quad\; +e^{\omega (t+\epsilon_1)}\big(1+e^{\omega h_n}\big)
    \big(2r\delta_2+\rho_1\big(r, b_{0,1}r\delta_2^{-1}\big)\epsilon_1+r\epsilon_2\big) \\
    &\quad\; +e^{\omega t}\big(Lre^{\omega(t+\epsilon_1)}\epsilon_2+\rho_2(\epsilon_1, \epsilon_2, h_n, r,t)\big)t
   \end{align*}
   for all $\delta_1, \delta_2>0$, $\epsilon_1, \epsilon_2\in (0,\epsilon_0]$, $r, t\geq 0$, $f\in\Lipb(r)$
   and $n\in\N$. 
  \item Suppose that Assumption~\ref{ass:rate}(iii) is valid. Then,
   \begin{align*}
    \|(S(t)f-I(\pi_n^t)f)^+\|_\kappa
    &\leq e^{\omega (t+h_n)}\big(2r\delta_1+\rho_1\big(r, b_{0,1}r\delta_1^{-1}\big)h_n\big) \\
    &\quad\; + e^{\omega(t+\epsilon_1)}\big(1+e^{\omega h_n}\big)
     \big(2r\delta_2+\rho_1\big(r, b_{0,1}r\delta_2^{-1}\big)\epsilon_1+r\epsilon_2\big) \\
    &\quad\; +e^{\omega t}\big(Lre^{\omega\epsilon_1}\epsilon_2+\rho_3(\epsilon_1, \epsilon_2, h_n, r,t)\big)t
   \end{align*}
   for all $\delta_1,\delta_2>0$, $\epsilon_1, \epsilon_2\in (0,\epsilon_0]$, $r, t\geq 0$, $f\in\Lipb(r)$
   and $n\in\N$. 
 \end{enumerate}
\end{theorem}

We have already commented on the roles of the variables $(\epsilon_1,\epsilon_2,h_n,r,t)$
appearing as arguments of the functions $\rho_2$ and $\rho_3$.\ While the right-hand side of
the two inequalities in the previous theorem typically decreases as $h_n$ decreases, the same
is not valid for the variables $(\delta_1,\delta_2,\epsilon_1,\epsilon_2)$.\ Indeed, the functions
$\rho_1,\rho_2$ and $\rho_3$ usually explode as $\delta_1,\delta_2,\epsilon_1,\epsilon_2\to 0$
but the right-hand side can only be small if the variables $(\delta_1,\delta_2,\epsilon_1,\epsilon_2)$
are small. Hence, in order to obtain an explicit and optimal convergence rate, one has to minimize
over these variables and the optimal choice will depend on $h_n$. In many applications,
this procedure leads to rates of the form
\[ -c_{r,t}^-h_n^{\gamma^-}\leq \big(S(t)f-I(\pi_n^t)f\big)\kappa\leq c_{r,t}^+h_n^{\gamma^+} \]
for all $r\geq 1$, $t\geq 0$, $f\in\Lipb(r)$ and $n\in\N$, where the constants $c_{r,t}^\pm\geq 0$
depend on $(r,t)$ but the exponents $\gamma^{\pm}\in (0,1]$ do not.

\begin{assumption} \label{ass:rate2}
 Suppose that the functions $\rho_1, \rho_2, \rho_3$ from Assumption~\ref{ass:rate}
 can be estimated as follows:
 \begin{enumerate}
  \item There exist  $a_1\colon\R_+\to\R$ and $a_2,p\ge 0$ with
   \[ \rho_1(x_1, x_2)\leq a_1(x_1) +a_2x_2^{p}
 	\quad\mbox{for all } x_1, x_2\geq 1. \]
  \item There exist $N^-\in\N$, $\theta_1^-,\ldots,\theta_{N^-}^-\colon\R_+^2\to\R_+$
   and $\alpha_1^-,\beta_1^-, \ldots,\alpha_{N^-}^-,\beta_{N^-}^-\geq 0$ with
  \[ \rho_2(\epsilon_1,\epsilon_2,h_n,r,t)
  	\leq\sum_{i=1}^{N^-} \theta_i^-(r,t)h_n^{\alpha_i^-}\epsilon_2^{-\beta_i^-} \]
  for all $r,t\geq 0$, $n\in\N$ and $\epsilon_1,\epsilon_2\in (0,\epsilon_0]$,
  with $\epsilon_1=\epsilon_2^{1+p}\geq h_n$.
 \item There exist $N^+\in\N$, $\theta_1^+,\ldots,\theta_{N^+}^+\colon\R_+^2\to\R_+$
   and $\alpha_1^+,\beta_1^-, \ldots,\alpha_{N^+}^+,\beta_{N^+}^+\geq 0$ with
  \[ \rho_3(\epsilon_1,\epsilon_2,h_n,r,t)
  	\leq\sum_{i=1}^{N^+} \theta^+_i(r,t)h_n^{\alpha_i^+}\epsilon_2^{-\beta_i^+} \]
  for all $r,t\geq 0$, $n\in\N$ and $\epsilon_1,\epsilon_2\in (0,\epsilon_0]$, 
  with $\epsilon_1=\epsilon_2^{1+p}\geq h_n$.
 \end{enumerate}
\end{assumption}

If we suppose that condition~(i) is valid, Corollary~\ref{cor:time} guarantees that the trajectories 
$t\mapsto S(t)f$ are $\frac{1}{1+p}$-H\"older continuous for all $f\in\Lipb$. This motivates the 
dependence between the time and space regularization parameters $\epsilon_1$ and $\epsilon_2$ 
in the conditions~(ii) and~(iii) which is clearly evident in the examples presented in Section~\ref{sec:examples}.

\begin{remark}
 We further comment on the consistency bounds in Assumption~\ref{ass:rate}.
 Suppose that Assumption~\ref{ass:rate2}(i) is valid. Then, for every $\epsilon=(\epsilon_1,\epsilon_2)$ 
 with $\epsilon_1,\epsilon_2>0$, $n\in\N$, $r\geq 0$, $t\geq h_n$, $f\in\Lipb(r)$ and $u(t):=S(t)f$, 
 it follows from Lemma~\ref{lem:deriv2} that 
 \begin{align*}
  &\partial_t u^{\epsilon}(t)-Au^{\epsilon}(t)-\frac{u^{\epsilon}(t)-I(h_n)u^{\epsilon}(t-h_n)}{h_n} \\
  &\leq\frac{I(h_n)u^{\epsilon}(t-h_n)-u^{\epsilon}(t-h_n)}{h_n}-Au^{\epsilon}(t)
  	+\partial_t u^{\epsilon}(t)-\frac{u^{\epsilon}(t)-u^{\epsilon}(t-h_n)}{h_n} \\
  &\leq\frac{I(h_n)u^{\epsilon}(t-h_n)-u^{\epsilon}(t-h_n)}{h_n}-Au^{\epsilon}(t)
  	+\frac{1}{2\kappa}\big(c_\kappa c_{r,T}\epsilon_1^\alpha+e^{\omega t}r\epsilon_2\big)
	b_{2,0}h_n\epsilon_1^{-2},
 \end{align*}
 where $c_{r,T}:=e^{\omega T}(2r+a_1(r)+a_2b_{0,1}^p r^p)$. 
 In order to verify Assumption~\ref{ass:rate2}(ii), it is therefore sufficient to provide the desired upper bound for 
 \begin{align*}
 &\frac{I(h_n)u^{\epsilon}(t-h_n)-u^{\epsilon}(t-h_n)}{h_n}-Au^{\epsilon}(t) \\
 &=\left(\frac{I(h_n)u^{\epsilon}(t-h_n)-u^{\epsilon}(t-h_n)}{h_n}-Au^{\epsilon}(t-h_n)\right) 
 	+Au^{\epsilon}(t-h_n)-Au^{\epsilon}(t).
 \end{align*}
 If $A$ is a differential operator of order $N$, the first and and second term on the right-hand
 side can typically be controlled by the first $N+1$ and $N$ spatial derivatives of the function
 $u^{\epsilon}$, respectively. The same argumentation applies for the lower bound.
\end{remark}

Optimizing over the parameters $(\delta_1, \delta_2, \epsilon_1,\epsilon_2)$ in
Theorem~\ref{thm:rate} leads to the following result which is useful to treat
the examples presented in Section~\ref{sec:examples}.

\begin{theorem} \label{thm:rate2}
 Suppose that Assumption~\ref{ass:rate} and Assumption~\ref{ass:rate2}(i) are satisfied. 
 \begin{enumerate}
  \item Suppose that  Assumption~\ref{ass:rate2}(ii) is valid. Then,
   \[ \|(S(t)f-I(\pi_n^t)f)^-\|_\kappa\leq c_{r,t}^- h_n^{\gamma^-} \]
   for all $r\geq 1$, $t\geq 0$, $f\in\Lipb(r)$ and $n\in\N$ with $h_n^{\gamma^-}\leq\epsilon_0$,
   where
   \[ \gamma^-:=\min\bigg\{\frac{1}{1+p},\frac{\alpha_1^-}{1+\beta_1^-}, \ldots,\frac{\alpha_{N^-}^-}{1+\beta_{N^-}^-}\bigg\} \]
   and, for $\epsilon_1^-:=h_0^{(1+p)\gamma^-}$ with $h_0:=\max_{n\in\N}h_n$,
   \begin{align*}
   	c_{r,t}^- &:=e^{\omega(t+h_0)}\big(2r+a_1(r)+a_2 b_{0,1}^p r^{p}\big) 
   	+e^{\omega(t+\epsilon_1^-)}(1+e^{\omega h_0})\big(3r+a_1(r)+a_2 b_{0,1}^p r^{p}\big) \\
   	&\quad\; 	+e^{\omega t}\Big(Lre^{\omega(t+\epsilon_1^-)}+\sum_{i=1}^{N^-}\theta_i^-(r,t)\Big)t.
   \end{align*}
  \item Suppose that  Assumption~\ref{ass:rate2}(iii) is valid. Then,
   \[ \|(S(t)f-I(\pi_n^t)f)^+\|_\kappa\leq c_{r,t}^+ h_n^{\gamma^+} \]
   for all $r\geq 1$, $t\geq 0$, $f\in\Lipb(r)$ and $n\in\N$ with $h_n^{\gamma^+}\leq\epsilon_0$,
   where 
   \[ \gamma^+:=\min\bigg\{\frac{1}{1+p},\frac{\alpha_1^+}{1+\beta_1^+}, \ldots,\frac{\alpha_{N^+}^+}{1+\beta_{N^+}^+}\bigg\} \]
   and, for $\epsilon_1^+:=h_0^{(1+p)\gamma^+}$ with $h_0:=\max_{n\in\N}h_n$,
     \begin{align*}
   	c_{r,t}^+ &:=e^{\omega(t+h_0)}\big(2r+a_1(r)+a_2 b_{0,1}^p r^{p}\big) 
   	+e^{\omega(t+\epsilon_1^+)}\big(1+e^{\omega h_0}\big)\big(3r+a_1(r)+a_2b_{0,1}^p r^{p}\big) \\
   	&\quad\; +e^{\omega t}\Big(Lre^{\omega \epsilon_1^+}+\sum_{i=1}^{N^+}\theta_i^+(r,t)\Big)t.
   \end{align*}
 \end{enumerate}
\end{theorem}

\subsection{Time regularity and consistency}
\label{sec:time}

Let $(I(t))_{t\geq 0}$ be a family of operators satisfying Assumption~\ref{ass:cher}
and denote by $(S(t))_{t\geq 0}$ be the corresponding semigroup from Theorem~\ref{thm:cher}.

\begin{lemma} \label{lem:deriv+}
 Let $r\colon\R_+\to\R_+$  be a non-decreasing function and let $u\colon\R_+\times\R^d\to\R$ 
 be a locally bounded measurable function with $u(t,\cdot)\in\Lipb(r(t))$ for all $t\geq 0$.
 Then, for every $t\geq 0$, $\epsilon=(\epsilon_1,\epsilon_2)$ with $\epsilon_1,\epsilon_2>0$, 
 $k\in \N_0$ and $l\in\N$,
 \[ \|\partial^k_t D^l u^{\epsilon}(t,\cdot)\|_\infty
 	\leq\d^{\frac{l}{2}}r(t+\varepsilon_1)b_{k,l-1}\epsilon_1^{-k}\epsilon_2^{1-l}.\]
\end{lemma}
\begin{proof}
 For every $t\geq 0$, $x\in\R^d$, $k\in \N_0$ and $\alpha, \beta, \gamma\in\N_0^d$ with 
 $\alpha=\beta+\gamma$ and $|\beta|=1$, so that $|\gamma|=|\alpha|-1$, we obtain
 \begin{align*}
  |\partial^k_t D^\alpha u^{\epsilon}(t,x)|
  &\leq\int_{\R_+\times\R^d}|D^\beta u(s+t,x+y)|\cdot|\partial^k_t D^\gamma\eta^{\epsilon}(s,y)|\,\d s\,\d y.
 \end{align*}
 Using that $u(s,\cdot)\in\Lipb(r(s))$ for all $s\geq 0$, it follows that $|D^\beta u(s+t,x+y)|\leq r(t+\epsilon_1)$ 
 for all $s\in[0,\epsilon_1]$, $t\geq 0$ and $x,y\in\R^d$ and therefore
 \begin{align*}
  |\partial^k_t D^\alpha u^{\epsilon}(t,x)| 
  &\leq r(t+\epsilon_1)\epsilon_1^{-k}\epsilon_2^{-|\gamma|} 
  	\int_{\R_+\times\R^d}|\partial^k_t D^\gamma\eta(s,y)|\,\d s\,\d y \\ 
  &\leq r(t+\epsilon_1)b_{k,|\alpha|-1}\epsilon_1^{-k}\epsilon_2^{1-|\alpha|} .
 \end{align*}
 Hence, for every $l\in\N$, equation~\eqref{const:b} implies
 \begin{align*}
  \|\partial^k_t D^l u^{\epsilon}(t,\cdot)\|_\infty
  &\leq\left(\sum\nolimits_{|\alpha|=l}\big(r(t+\epsilon_1)\epsilon_1^{-k}\epsilon_2^{1-|\alpha|} b_{k,|\alpha|-1}\big)^2\right)^{1/2} \\
  &\leq d^{\frac{l}{2}} r(t+\epsilon_1)b_{k,l-1}\epsilon_1^{-k}\epsilon_2^{1-l}.\qedhere
 \end{align*}
\end{proof}

This leads to the following time regularity estimates.

\begin{theorem} \label{thm:time}
 Suppose that Assumption~\ref{ass:rate}(i) is satisfied. Then,
 \begin{align*}
  \|I(h_n)^k f-I(h_n)^l f\|_\kappa
  &\leq e^{\omega T}\big(2r\delta+\rho_1(r, b_{0,1}r\delta^{-1})|k-l|h_n\big), \\
  \|S(s)f-S(t)f\|_\kappa
  &\leq e^{\omega T}\big(2r\delta+\rho_1(r, b_{0,1}r\delta^{-1})|s-t|\big)
 \end{align*}
 for all $r,T\geq 0$, $\delta>0$, $f\in\Lipb(r)$, $k, l, n\in\N$ with $kh_n,lh_n\in [0,T]$ and $s,t\in [0,T]$.
\end{theorem}
\begin{proof}
 Fix $r, T\geq 0$. It follows from Assumption~\ref{ass:rate}(i) and~\cite[Lemma~2.8]{BK22+} that
 \begin{equation} \label{eq:time1}
  \|I(h_n)^k f-I(h_n)^l f\|_\kappa\leq e^{\omega T}\rho_1\big(d^{-\frac{1}{2}}\|D f\|_\infty, d^{-1}\|D^2 f\|_\infty\big)|k-l|h_n
 \end{equation}
 for all $k,l, n\in\N$ with $kh_n,lh_n\in [0,T]$ and $f\in\Cbi$. Furthermore, by~\cite[Lemma~2.7]{BK22+},
 \begin{equation} \label{eq:time2}
  \|I(h_n)^k f-I(h_n)^k g\|_\kappa\leq e^{\omega T}\|f-g\|_\kappa
 \end{equation}
 for all $k,l, n\in\N$ with $kh_n,lh_n\in [0,T]$ and $f,g\in\Ck$.
 Let $\delta>0$, $f\in\Lipb(r)$ and
 \[ f_\delta(x):=\int_{\R_+\times\R^d}f(x+y)\eta^{(1,\delta)}(s,y)\,\d s\,\d y \quad\mbox{for all } x\in\R^d. \]
 It holds $\|f-f_\delta\|_\kappa\leq r\delta$ and Lemma~\ref{lem:deriv+} implies
$ \|D^l f_\delta\|_\infty\leq d^{\frac{l}{2}} b_{0,l-1} r \delta^{1-l}$ for all $l\in\N$.
 It follows from inequality~\eqref{eq:time1} and inequality~\eqref{eq:time2} that
 \begin{align*}
  &\|I(h_n)^k f-I(h_n)^l f\|_\kappa\\
  &\leq\|I(h_n)^k f-I(h_n)^k f_\delta\|_\kappa+\|I(h_n)^k f_\delta-I(h_n)^l f_\delta\|_\kappa
  +\|I(h_n)^l f_\delta-I(h_n)^l f\|_\kappa \\
  &\leq 2e^{\omega T}\|f-f_\delta\|_\kappa
  	+e^{\omega T}\rho_1\big(d^{-\frac{1}{2}}\|D f_\delta\|_\infty, d^{-1}\|D^2 f_\delta\|_\infty\big)|k-l|h_n \\
  &\leq e^{\omega T}\big(2r\delta+\rho_1\big(r, b_{0,1}r\delta^{-1}\big)|k-l|h_n\big)
 \end{align*}
 for all $k,l, n\in\N$ with $kh_n,lh_n\in [0,T]$. In addition, equation~\eqref{eq:cher} implies
 \[ \|S(s)f-S(t)f\|_\kappa
 	\leq e^{\omega T}\big(2r\delta+\rho_1\big(r, b_{0,1}r\delta^{-1}\big)|s-t|\big)
 	\quad\mbox{for all } s, t\in [0,T]. \qedhere \]
\end{proof}

If the function $\rho_1$ can be estimated as in inequality~\eqref{eq:rho} below, optimizing over
the parameter $\delta$ in the previous theorem leads to H\"older continuity in time.
This observation will be very useful in the applications presented in Section~\ref{sec:examples}.

\begin{corollary} \label{cor:time}
 Suppose that there exist $a_1\colon\R_+\to\R$ and $a_2,p\ge 0$ with
 \begin{equation} \label{eq:rho}
  \rho_1(x_1, x_2)\leq a_1(x_1)+a_2x_2^{p}
  \quad\mbox{for all } x_1, x_2\geq 1.
 \end{equation}
 Then, for every $r\geq 1$, $T\geq 0$, $f\in\Lipb(r)$,
 $k,l, n\in\N$ satisfying $kh_n,lh_n\in [0,T]$ and $|k-l|h_n\leq 1$ and $s,t\in [0,T]$ with $|s-t|\leq 1$,
 \begin{align*}
  \|I(h_n)^k f-I(h_n)^l f\|_\kappa &\leq c_{r,T}(|k-l|h_n)^\alpha, \\
  \|I(\pi_n^s)f-I(\pi_n^t)f\|_\kappa &\leq c_{r,T}(|s-t|+h_n)^\alpha, \\
  \|S(s)f-S(t)f\|_\kappa &\leq c_{r,T}|s-t|^\alpha,
 \end{align*}
 where $c_{r,T}:=e^{\omega T}\big(2r+a_1(r)+a_2 b_{0,1}^pr^p\big)$ and $\alpha:=\frac{1}{1+p}$.
\end{corollary}
\begin{proof}
 It follows from Theorem~\ref{thm:time} and inequality~\eqref{eq:rho} that
 \[ \|I(h_n)^k f-I(h_n)^l f\|_\kappa
	\leq e^{\omega T}\big(2r\delta
	+\big(a_1(r)+a_2b_{0,1}^pr^{p}\delta^{-p}\big)|k-l|h_n\big) \]
 for all $\delta\in (0,1]$, $r\geq 1$, $T\geq 0$, $f\in\Lipb(r)$ and $k,l, n\in\N$ satisfying $kh_n,lh_n\in [0,T]$
 and $|k-l|h_n\leq 1$. Choosing $\delta:=(|k-l|h_n)^\alpha\in (0,1]$ yields
 \[ \|I(h_n)^k f-I(h_n)^l f\|_\kappa\leq c_{r,T}(|k-l|h_n)^\alpha. \]
 Furthermore, it follows from $|k_n^s-k_n^t|h_n\leq |s-t|+h_n$ that
 \[ \|I(\pi_n^s)f-I(\pi_n^t)f\|_\kappa\leq c_{r,T}(|s-t|+h_n)^\alpha \]
 and equation~\eqref{eq:cher} implies $\|S(s)f-S(t)f\|_\kappa\leq c_{r,T}|s-t|^\alpha$.
\end{proof}

In the following two results, we provide regularity estimates for $u(t):=S(t)f$ and $u_n(t)f:=I(\pi_n^t)f$ 
which are useful to verify Assumption~\ref{ass:rate}(ii) and~(iii).

\begin{lemma} \label{lem:deriv}
 Let $r\geq 0$ and $f\in\Lipb(r)$. Define $u(t):=S(t)f$ and $u_n(t)f:=I(\pi_n^t)f$ for all $t\geq 0$ and $n\in\N$.
 Then, for every $t\geq 0$, $\epsilon=(\epsilon_1,\epsilon_2)$ with $\epsilon_1,\epsilon_2>0$ and $l, n\in\N$,
 \[ \max\big\{\|D^l u^{\epsilon}(t)\|_\infty,\|D^l u_n^{\epsilon}(t)\|_\infty\big\}
 	\leq d^{\frac{l}{2}} r e^{\omega (t+\epsilon_1)} b_{0,l-1}\epsilon_2^{1-l}.\]
\end{lemma}
\begin{proof}
 By Theorem~\ref{thm:cher}(vi) and \cite[Lemma~2.7]{BK22+}, it holds $u(t),u_n(t)\in\Lipb(e^{\omega t}r)$ 
 for all $t\geq 0$ and $n\in\N$. The claim follows from Lemma~\ref{lem:deriv+}.
 \end{proof}

\begin{lemma} \label{lem:deriv2}
 Suppose that there exist $a_1\colon\R_+\to\R$ and $a_2,p\geq 0$ with
 \[ \rho_1(x_1, x_2)\leq a_1(x_1)+a_2x_2^p \quad\mbox{for all } x_1, x_2\geq 1.\]	
 For fixed $r\geq 1$ and $T\geq 0$, we denote by $c_{r,T}$ and $\alpha$ the parameters from
 Corollary~\ref{cor:time}. Let $f\in\Lipb(r)$ and define $u(t):=S(t)f$ and $u_n(t):=I(\pi_n^t)f$ for all
 $t\geq 0$ and $n\in\N$. Then, for every $t\in [0,T]$, $\epsilon=(\epsilon_1,\epsilon_2)$ with 
 $\epsilon_1,\epsilon_2>0$, $k,l\in\N_0$ and $n\in\N$,
 \begin{align*}
  \|\partial_t^k D^l u^{\epsilon}(t)\|_\kappa
  &\leq d^{\frac{l}{2}}\big(c_\kappa c_{r,T}\epsilon_1^\alpha+e^{\omega t}r\epsilon_2\big)
  	b_{k,l} \epsilon_1^{-k}\epsilon_2^{-l}, \\
  \|\partial_t^k D^l u_n^{\epsilon}(t)\|_\kappa
  &\leq d^{\frac{l}{2}}\big(c_\kappa c_{r,T}(\epsilon_1+h_n)^\alpha+e^{\omega t}r\epsilon_2\big)
  	b_{k,l}\epsilon_1^{-k}\epsilon_2^{-l}.
 \end{align*}
 In addition, for every $n\in\N$, $t\in [h_n,T]$, $\epsilon=(\epsilon_1,\epsilon_2)$ with 
 $\epsilon_1,\epsilon_2>0$ and $k,l\in\N_0$,
 \begin{align*}
  \left\|\partial_t u^{\epsilon}(t)-\frac{u^{\epsilon}(t)-u^\epsilon(t-h_n)}{h_n}\right\|_\kappa
  &\leq\frac{1}{2}\big(c_\kappa c_{r,T}\epsilon_1^\alpha+e^{\omega t}r\epsilon_2\big)
  	b_{2,0} h_n\epsilon_1^{-2}, \\
  \left\|\partial_t u_n^{\epsilon}(t)-\frac{u_n^{\epsilon}(t)-u_n^\epsilon(t-h_n)}{h_n}\right\|_\kappa
  &\leq\frac{1}{2}\big(c_\kappa c_{r,T}(\epsilon_1+h_n)^\alpha+e^{\omega t}r\epsilon_2\big)
  	b_{2,0} h_n\epsilon_1^{-2}.
 \end{align*}
\end{lemma}
\begin{proof}
 For every $t\geq 0$, $x\in\R^d$, $k\in\N_0$ and $\gamma\in\N_0^d$, it follows from
 \[ \int_{\R_+\times\R^d}\partial_t^k D^\gamma u^\epsilon(s+t,x+y)\,\d s\,\d y=0 \]
 that
 \begin{align*}
  \partial_t^k D^\gamma u^{\epsilon}(t,x)
  &=\int_{\R_+\times\R^d}u(s+t,x+y)\partial_t^k D^\gamma\eta^{\epsilon}(s,y)\,\d s\,\d y \\
  &=\int_{\R_+\times\R^d}\big(u(s+t,x+y)-u(t,x+y)\big)\partial_t^k D^\gamma\eta^{\epsilon}(s,y)\,\d s\,\d y \\
  &\quad\; +\int_{\R_+\times\R^d}\big(u(t,x+y)-u(t,x)\big)\partial_t^k D^\gamma\eta^{\epsilon}(s,y)\,\d s\,\d y.
 \end{align*}
 Condition~\eqref{eq:kappa}, Corollary~\ref{cor:time} and Theorem~\ref{thm:cher}(vi) imply
 \begin{align*}
  |\partial_t^k D^\gamma u^{\epsilon}(t,x)|\kappa(x)
  &\leq\big(c_\kappa c_{r,T}\epsilon_1^\alpha+e^{\omega t}r\epsilon_2\big)
  	\int_{\R_+\times\R^d}|\partial_t^k D^\gamma\eta^{\epsilon}(s,y)|\,\d s\,\d y \\
  &\leq\big(c_\kappa c_{r,T}\epsilon_1^\alpha+e^{\omega t}r\epsilon_2\big)
  	b_{k,|\gamma|}\epsilon_1^{-k}\epsilon_2^{-|\gamma|}
 \end{align*}
 for all $t\geq 0$ and $x\in\R^d$. Hence, for every $l\in\N_0$, we obtain
 \[ \|\partial_t^k D^l u^{\epsilon}(t)\|_\kappa
 \leq d^\frac{l}{2}\big(c_\kappa c_{r,T}\epsilon_1^\alpha+e^{\omega t}r\epsilon_2\big)
 b_{k,l} \epsilon_1^{-k}\epsilon_2^{-l}.\]
 Furthermore,
 \begin{align*}
  \left\|\partial_t u^{\epsilon}(t)-\frac{u^{\epsilon}(t)-u^{\epsilon}(t-h)}{h}\right\|_\kappa
  &\leq\frac{1}{h}\int_0^h \|\partial_t u^{\epsilon}(t)-\partial_t u^{\epsilon}(t-s_1)\|_\kappa\,\d s_1 \\
  &\leq\frac{1}{h}\int_0^h\int_0^{s_1}\|\partial_t^2 u^{\epsilon}(t-s_2)\|_\kappa\,\d s_2\,\d s_1 \\
  &\leq\frac{1}{2}\big(c_\kappa c_{r,T}\epsilon_1^\alpha+e^{\omega t}r\epsilon_2\big)
 	 b_{2,0} h\epsilon_1^{-2}.
 \end{align*}
 Similarly, one can estimate the derivatives of $u_n^{\epsilon}(t)$.
\end{proof}

\subsection{Extension to unbounded functions}
\label{sec:unbounded}

Let $\Lip$ and $\Ci$ be the spaces of all Lipschitz continuous and infinitely differentiable 
functions $f\colon\R^d\to\R$, respectively. For every $r\geq 0$, the set $\Lip(r)$ consists 
of all $r$-Lipschitz functions $f\colon\R^d\to\R$. Let $\Lip\subset\Ck$ which means that 
the function $\nicefrac{1}{\kappa}$ grows at least linearly at infinity. In particular,
we do not provide rates w.r.t. the supremum norm for arbitrary Lipschitz continuous functions 
since the choice $\kappa\equiv 1$ is excluded.\ While extending the results to unbounded 
functions comes with the cost of weakening the norm, the rates are still uniform on
bounded sets $B_{\R^d}(R)$ with a constant depending linearly on $R\geq 0$. 
The next assumption generalizes properties that were previously only required for bounded 
functions to arbitrary Lipschitz continuous functions. Let $(I(t))_{t\geq 0}$ be a family of
operators satisfying Assumption~\ref{ass:cher}.

\begin{assumption} \label{ass:rate3} \Newline
 \begin{enumerate}
  \item There exist $t_0>0$, $\epsilon_0\in (0,1]$ and $L\geq 0$ with
   \[ \|I(t)(\tau_x f)-\tau_x I(t)f\|_\kappa\leq Lrt|x| \]
   for all $t\in [0,t_0]$, $x\in B_{\R^d}(\epsilon_0)$, $r\geq 0$ and $f\in\Lip(r)$.
  \item The limit $I'(0)f\in\Ck$ exists for all $f\in\Lip\cap\Ci$ such that the partial derivatives
   of any order are bounded.
  \item It holds $I(t)\colon\Lip(r)\to\Lip(e^{\omega t}r)$ for all $r,t\geq 0$.
  \item There exists a function $\rho_1\colon\R_+^2\to\R_+$ with
   \[ \|I(h_n)f-f\|_\kappa\leq\rho_1\big(d^{-\frac{1}{2}}\|D f\|_\kappa, d^{-1}\|D^2 f\|_\kappa\big)h_n \]
   for all $n\in\N$ and $f\in\Lip\cap\Ci$ such that the partial derivatives of any order are bounded.
   The function $\rho_1$ is non-decreasing in both arguments.
  \item There exists a function $\rho_2\colon\R_+^5\to\R_+$ with
   \[ \partial_t u^\epsilon(t)-Au^\epsilon(t)-\frac{u^\epsilon(t)-I(h_n)u^\epsilon(t-h_n)}{h_n}
    \leq\frac{\rho_2(\epsilon_1, \epsilon_2, h_n, r,t)}{\kappa} \]
   for all $\epsilon=(\epsilon_1,\epsilon_2)$ with $\epsilon_1, \epsilon_2\in (0,\epsilon_0]$,
   $n\in\N$, $r\geq 0$, $t\geq h_n$ and $f\in\Lip(r)$, where $u(t):=S(t)f$ and $u^\epsilon:=u*\eta^\epsilon$. 
  \item There exists a function $\rho_3\colon\R_+^5\to\R_+$ with
   \[ \partial_t u_n^\epsilon(t)-Au_n^\epsilon(t)-\frac{u_n^\epsilon(t)-I(h_n)u_n^\epsilon(t-h_n)}{h_n}
    \geq -\frac{\rho_3(\epsilon_1, \epsilon_2, h_n, r,t)}{\kappa} \]
   for all $\epsilon=(\epsilon_1,\epsilon_2)$ with $\epsilon_1, \epsilon_2\in (0,\epsilon_0]$,
   $n\in\N$, $r\geq 0$, $t\geq h_n$ and $f\in\Lip(r)$, where $u_n(t):=I(\pi_n^t)f$ and
   $u_n^\epsilon:=u_n*\eta^\epsilon$. Moreover, the function $\rho_3$ is non-decreasing 
   in the last argument.
 \end{enumerate}
\end{assumption}

Let $r\geq 0$, $f\in\Lip(r)$, $\delta>0$ and 
\[ f_\delta(x):=\int_{\R^d}f(x+y)\eta^{1,\delta}(y)\,\d y  \quad\mbox{for all } x\in\R^d.\]
Then, for every $\alpha,\beta,\gamma\in\N_0^d$ with $\alpha=\beta+\gamma$ and
$|\beta|=1$ and $x\in\R^d$,
\[ |D^\alpha f_\delta(x)|\leq\int_{\R^d}|D^\beta f(x-y)|\cdot |D^\gamma\eta^{1,\delta}(y)|\,\d y
	\leq r\int_{\R^d}|D^\gamma\eta^{1,\delta}(y)|\,\d y \]
which shows that $f_\delta\in\Lip\cap\Ci$ has bounded partial derivatives of any order.

\begin{theorem} \label{thm:rate3}
 Let $(I(t))_{t\geq 0}$ be a family of operators satisfying Assumption~\ref{ass:cher} and
 Assumption~\ref{ass:rate3}(i)-(iv). Denote by $(S(t))_{t\geq 0}$ the corresponding
 semigroup from Theorem~\ref{thm:cher}.
 \begin{enumerate}
  \item Suppose that Assumption~\ref{ass:rate3}(v) is valid. Then, 
   \begin{align*}
    \|(S(t)f-I(\pi_n^t)f)^-\|_\kappa
    &\leq e^{\omega (t+h_n)}\big(2r\delta_1+\rho_1\big(r, b_{0,1}r\delta_1^{-1}\big)h_n\big) \\
    &\quad\; +e^{\omega (t+\epsilon_1)}(1+e^{\omega h_n})
    \big(2r\delta_2+\rho_1\big(r, b_{0,1}r\delta_2^{-1}\big)\epsilon_1+r\epsilon_2\big) \\
    &\quad\; +e^{\omega t}\big(Lre^{\omega(t+\epsilon_1)}\epsilon_2+\rho_2(\epsilon_1, \epsilon_2, h_n, r,t)\big)t
   \end{align*}
   for all $\delta_1, \delta_2>0$, $\epsilon_1, \epsilon_2\in (0,\epsilon_0]$, $r, t\geq 0$, $f\in\Lip(r)$
   and $n\in\N$. 
  \item Suppose that Assumption~\ref{ass:rate3}(vi) is valid. Then, 
   \begin{align*}
    \|(S(t)f-I(\pi_n^t)f)^+\|_\kappa
    &\leq e^{\omega (t+h_n)}\big(2r\delta_1+\rho_1\big(r, b_{0,1}r\delta_1^{-1}\big)h_n\big) \\
    &\quad\; + e^{\omega(t+\epsilon_1)}\big(1+e^{\omega h_n}\big)
     \big(2r\delta_2+\rho_1\big(r, b_{0,1}r\delta_2^{-1}\big)\epsilon_1+r\epsilon_2\big) \\
    &\quad\; +e^{\omega t}\big(Lre^{\omega\epsilon_1}\epsilon_2+\rho_3(\epsilon_1, \epsilon_2, h_n, r,t)\big)t
   \end{align*}
   for all $\delta_1,\delta_2>0$, $\epsilon_1, \epsilon_2\in (0,\epsilon_0]$, $r, t\geq 0$, $f\in\Lip(r)$
   and $n\in\N$. 
 \end{enumerate}
\end{theorem}

\section{Proof of the main results}
\label{sec:proof}

Throughout this section, let $(I(t))_{t\geq 0}$ be a family of operators satisfying Assumption~\ref{ass:cher}
and let $(S(t))_{t\geq 0}$ be the corresponding semigroup from Theorem~\ref{thm:cher}.

\subsection{Comparison principles}

The following lemma is similar to~\cite[Lemma~3.2]{BJ07} but here we give a simpler proof that
uses the recursive structure of the approximation scheme. Furthermore, the result is valid for
unbounded functions.

\begin{lemma} \label{lem:BJ}
 Let $h,T\geq 0$ and let $u,v\colon [0,T]\to\Ck$ and $f,g\colon [h,T]\to\Ck$ be functions
 with
 \begin{equation} \label{eq:BJ}
  \frac{u(t)-I(h)u(t-h)}{h}\leq f(t) \quad\mbox{and}\quad \frac{v(t)-I(h)v(t-h)}{h}\geq g(t)
 \end{equation}
 for all $t\in [h,T]$. Then, for every $t\in [0,T]$,
 \[ \|(u(t)-v(t))^+\|_\kappa\leq e^{\omega t}\bigg(\sup_{s\in [0,h)}\|(u(s)-v(s))^+\|_\kappa
 	+t\sup_{s\in [h, T]}\|(f(s)-g(s))^+\|_\kappa\bigg). \]
\end{lemma}
\begin{proof}
 First, we assume that $u(t)\leq v(t)$ for all $t\in [0,h)$ and $f(t)\leq g(t)$ for all $t\in [h,T]$.
 For every $t\in [h,T]$, inequality~\eqref{eq:BJ} implies $u(t)-I(h)u(t-h)\leq v(t)-I(h)v(t-h)$.
 Furthermore, we use Lemma~\ref{lem:lambda}, the monotonicity of $I(h)$ and $t-h\in [0,h)$
 to obtain
 \begin{align*}
  u(t)-v(t) &\leq I(h)u(t-h)-I(h)v(t-h) \\
  &\leq\tfrac{1}{2}I(h)\big(2(u(t-h)-v(t-h))+v(t-h)\big)-\tfrac{1}{2}I(h)v(t-h) \\
  &\leq\tfrac{1}{2}I(h)v(t-h)-\tfrac{1}{2}I(h)v(t-h)=0
 \end{align*}
 for all $t\in [h,2h)$. By induction, one can show that $u(t)\leq v(t)$ for all $t\in [0,T]$.

 Second, for the general case, we consider the function
 \[ w(t):=v(t)+e^{\omega t}\bigg(\sup_{s\in [0,h)}\|(u(s)-v(s))^+\|_\kappa
 	+t\sup_{s\in [h, T]}\|(f(s)-g(s))^+\|_\kappa\bigg)\frac{1}{\kappa}. \]
 It holds $u(t)\leq (u(t)-v(t))^+ +v(t)\leq w(t)$ for all $t\in [0,h)$. In addition, for every $t\in [h, T]$,
 it follows from Assumption~\ref{ass:cher}(iii) and Corollary~\ref{cor:kappa} that
 \[ I(h)w(t-h)\leq I(h)v(t-h)+\frac{e^{\omega t}(a+b(t-h))}{\kappa}, \]
 where $a:=\sup_{s\in [0,h)}\|(u(s)-v(s))^+\|_\kappa$ and $b:=\sup_{s\in [h, T]}\|(f(s)-g(s))^+\|_\kappa$.
 Hence,
 \[ f(t) \leq g(t)+(f(t)-g(t))^+\leq\frac{v(t)-I(h)v(t-h)}{h}+\frac{b}{\kappa}
 	\leq\frac{w(t)-I(h)w(t-h)}{h} \]
 for all $t\in [h, T]$. It follows from the first part of the proof that $u(t)\leq w(t)$ and thus
 \[ \|(u(t)-v(t))^+\|_\kappa\leq e^{\omega t}(a+bt) \quad\mbox{for all } t\in [0,T]. \qedhere \]
\end{proof}

The following definition and comparison principle are from~\cite{BDKN22}. For every sequence
$(f_n)_{n\in\N}\subset\Ck$ with $\sup_{n\in\N}\|f_n^+\|_\kappa<\infty$ and $x\in\R^d$,
we define
 \[ \Big(\Glimsup_{n\to\infty}f_n\Big)(x)
 	:=\sup\Big\{\limsup_{n\to\infty}f_n(x_n)\colon
 	(x_n)_{n\in\N}\subset\R^d\mbox{ with } x_n\to x\Big\}\in [-\infty, \infty). \]
Moreover, for every family $(f_t)_{t\in [T_1, T_2]}\subset\Ck$ with $T_1\leq T_2$ and
$\sup_{t\in [T_1, T_2]}\|f_t^+\|_\kappa<\infty$,
\[ \Glimsup_{s\to t}f_s:=\sup\Big\{\Glimsup_{n\to\infty}f_{s_n}\colon s_n\to t\Big\}
	\quad\mbox{for all } t\in [T_1, T_2]. \]
The $\Gamma$-limit superior is always an upper semicontinuous function $f\colon\R^d\to [-\infty, \infty)$
with $\|f^+\|_\kappa<\infty$.

\begin{definition}
 The upper Lipschitz set $\L^S_+$ consists of all $f\in\Ck$ such that there
 exist $c\geq 0$ and $t_0>0$ with
 \[ \|(S(t)f-f)^+\|_\kappa\leq ct \quad\mbox{for all } t\in [0,t_0]. \]
 Moreover, for every $f\in\L^S_+$, the upper $\Gamma$-generator is defined by
 \[ \AG^+ f:=\Glimsup_{h\downarrow 0}\frac{S(h)f-f}{h}. \]
\end{definition}

\begin{theorem} \label{thm:comp}
 Let $T_2, T_1\geq 0$ with $T_1\leq T_2$ and $f\in\L^S_+$. Let $v\colon [T_1,T_2]\to\L^S_+$
 be a function with $S(T_1)f\leq v(T_1)$, $\sup_{t\in [T_1,T_2]}\|v(t)\|_\kappa<\infty$ and
 $\Glimsup_{s\to t}v(s)\leq v(t)$ for all $t\in [T_1,T_2]$. Suppose that, for every $t\in [T_1,T_2)$,
 \begin{align}
  \limsup_{h\downarrow 0}\Big\|\Big(\frac{v(t+h)-v(t)}{h}\Big)^-\Big\|_\kappa &<\infty, \label{eq:comp1} \\
  \Glimsup_{h\downarrow 0}\Big(\AG^+ v(t)-\frac{v(t+h)-v(t)}{h}\Big) &\leq 0. \label{eq:comp2}
 \end{align}
 Then, it holds $S(t)f\leq v(t)$ for all $t\in [T_1,T_2]$.
\end{theorem}

\subsection{Convolution and $\Gamma$-supersolutions}

In the sequel, we fix $\epsilon_1, \epsilon_2\in (0, \epsilon_0]$ and a probability measure $\mu$
on $\B(\R_+\times\R^d)$ with $\supp(\mu)\subset [0,\epsilon_1]\times B_{\R^d}(\epsilon_2)$.
For a locally bounded measurable function $u\colon\R_+\times\R^d\to\R$, we define
\[ (u*\mu)(t,x):=\int_{\R_+\times\R^d}u(s+t, x+y)\,\mu(\d s\times\d y)
	\quad\mbox{for all } (t,x)\in\R_+\times\R^d. \]
Moreover, for a function $u\colon\R_+\to\Ck$ and $t\geq 0$ such that there exist $c\geq 0$
and $h_0>0$ with $\|(u(t+h)-u(t))^+\|_\kappa\leq ch$ for all $h\in [0,h_0]$, we define
\[ \partial_t^\Gamma u(t):=\Glimsup_{h\downarrow 0}\frac{u(t+h)-u(t)}{h}. \]
For functions $u\colon\R_+\times\R^d\to \R$, we subsequently write $u(t):=u(t, \cdot)$.

\begin{lemma} \label{lem:super}
 Let $r\geq 0$, $f\in\Lipb(r)$ and $u(t):=S(t)f$ for all $t\geq 0$. Suppose that,
 for every $t\geq 0$, there exist $c\geq 0$ and $h_0>0$ such that the function $v:=u*\mu$ satisfies
 $\|(v(t+h)-v(t))^+\|_\kappa\leq ch$ for all $h\in (0,h_0]$. Then, it holds $v(t)\in\L^S_+$ and
 \[ \partial_t^\Gamma v(t)\geq\AG^+ v(t)-\frac{Lre^{\omega(t+\epsilon_1)}\epsilon_2}{\kappa}
 	\quad\mbox{for all } t\geq 0. \]
\end{lemma}
\begin{proof}
 For every $h, t\geq 0$ and $x\in\R^d$, Lemma~\ref{lem:Jensen} and Theorem~\ref{thm:cher}(vi)
 imply
 \begin{align*}
  \big(S(h)v(t, \,\cdot\,)\big)(x)
  &\leq\int_{\R_+\times\R^d}\big(S(h)u(s+t, \,\cdot\,+y)\big)(x)\,\mu(\d s\times\d y) \\
  &\leq\int_{\R_+\times\R^d}\big(S(h)u(s+t, \,\cdot\,)\big)(x+y)\,\mu(\d s\times\d y)
  	+\frac{Lrhe^{\omega(t+\epsilon_1+2h)}\epsilon_2}{\kappa(x)} \\
  &=\int_{\R_+\times\R^d}u(s+t+h, x+y)\,\mu(\d s\times\d y)
  	+\frac{Lrhe^{\omega(t+\epsilon_1+2h)}\epsilon_2}{\kappa(x)} \\
  &=v(t+h, x)+\frac{Lrhe^{\omega(t+\epsilon_1+2h)}\epsilon_2}{\kappa(x)}.
 \end{align*}
 Moreover, for every $t\geq 0$, there exist $c\geq 0$ and $h_0>0$ with $\|(v(t+h)-v(t))^+\|_\kappa\leq ch$
 for all $h\in (0,h_0]$. The previous estimate yields
 $\|(S(h)v(t)-v(t))^+\|_\kappa\leq ch+Lrhe^{\omega(t+\epsilon_1)}\epsilon_2$ for all $h\in (0,h_0]$
 and
 \[ \partial_t^\Gamma v(t)\geq\AG^+ v(t)-\frac{Lre^{\omega(t+\epsilon_1)}\epsilon_2}{\kappa}. \qedhere \]
\end{proof}

The conditions of the previous lemma are clearly satisfied if $\mu$ has a smooth density
and thus $v\in\Cbi([0,T]\times\R^d)$ for all $T\geq 0$. Then, it holds $v(t)\in\Cbi\subset D(A)$ and
\[ \partial_t v(t)\geq Av(t)-\frac{Lre^{\omega(t+\epsilon_1)}\epsilon_2}{\kappa}
 	\quad\mbox{for all } t\geq 0. \]
If $(S(t))_{t\geq 0}$ translation-invariant, i.e., $S(t)(\tau_x f)=\tau_x S(t)f$ for all $t\geq 0$,
$f\in\Ck$ and $x\in\R^d$, we further obtain
\[  \partial_t^\Gamma v(t)\geq\AG^+ v(t) \quad\mbox{for all } t\geq 0 \]
which means that the function $v$ is a $\Gamma$-supersolution. The following result is a
discrete version of Lemma~\ref{lem:super}.

\begin{lemma} \label{lem:super2}
 Let $n\in\N$, $r\geq 0$, $f\in\Lipb(r)$ and $u_n(t):=I(\pi_n^t)f$ for all $t\geq 0$. Then,
 the function $v_n:=u_n*\mu$ satisfies
 \[ \frac{v_n(t)-I(h_n)v_n(t-h_n)}{h_n}\geq -\frac{Lre^{\omega (t+\epsilon_1)}\epsilon_2}{\kappa}
 	\quad\mbox{for all } t\geq h_n. \]
\end{lemma}
\begin{proof}
 For every $t\geq h_n$ and $x\in\R^d$, it follows from Lemma~\ref{lem:Jensen}, Assumption~\ref{ass:cher}(iv)
 and Assumption~\ref{ass:cher}(vii) that
 \begin{align*}
  \big(I(h_n)v_n(t, \,\cdot\,)\big)(x)
  &\leq\int_{\R_+\times\R^d}\big(I(h_n)u_n(s+t, \,\cdot\,+y)\big)(x)\,\mu(\d s\times\d y) \\
  &\leq\int_{\R_+\times\R^d}\big(I(h_n)u_n(s+t, \,\cdot\,)\big)(x+y)\,\mu(\d s\times\d y)
  	+\frac{Lrh_ne^{\omega(t+\epsilon_1)}\epsilon_2}{\kappa(x)} \\
  &=\int_{\R_+\times\R^d}u_n(s+t+h_n, x+y)\,\mu(\d s\times\d y)
   +\frac{Lrh_ne^{\omega(t+\epsilon_1)}\epsilon_2}{\kappa(x)} \\
  &=v_n(t+h_n,x)+\frac{Lrh_ne^{\omega(t+\epsilon_1)}\epsilon_2}{\kappa(x)}. \qedhere
 \end{align*}
\end{proof}

\subsection{Lower bound}

\begin{proof}[Proof of Theorem~\ref{thm:rate}(i)]
 Let $r\geq 0$, $f\in\Lipb(r)$ and define $u(t):=S(t)f$ for all $t\geq 0$. 
 Furthermore, we fix $\delta=(\delta_1,\delta_2)$ with $\delta_1, \delta_2>0$ and 
 $\epsilon=(\epsilon_1,\epsilon_2)$ with $\epsilon_1, \epsilon_2\in (0,\epsilon_0]$.
 For every $n\in\N$ and $t\geq h_n$, Assumption~\ref{ass:rate}(ii) and Lemma~\ref{lem:super}
 imply
 \begin{align*}
  \frac{u^{\epsilon}(t)-I(h_n)u^{\epsilon}(t-h_n)}{h_n}
  &\geq\partial_t u^{\epsilon}(t)-Au^{\epsilon}(t)-\frac{\rho_2(\epsilon_1, \epsilon_2, h_n, r,t)}{\kappa} \\
  &\geq -\frac{Lre^{\omega(t+\epsilon_1)}\epsilon_2+\rho_2(\epsilon_1, \epsilon_2, h_n, r,t)}{\kappa}.
 \end{align*}
 Since the function $u_n(t):=I(\pi_n^t)f$ satisfies $u_n(t)=I(h_n)u_n(t-h_n)$, Lemma~\ref{lem:BJ} yields
 \begin{align}
  &\|(u_n(t)-u^{\epsilon}(t))^+\|_\kappa \nonumber \\
  &\leq e^{\omega t}\bigg(\sup_{s\in [0, h_n)}\|(u_n(s)-u^{\epsilon}(s))^+\|_\kappa
  +\big(Lre^{\omega(t+\epsilon_1)}\epsilon_2+\rho_2(\epsilon_1, \epsilon_2, h_n, r,t)\big)t\bigg) \label{eq:lb1}
 \end{align}
 for all $n\in\N$ and $t\geq h_n$. Furthermore, we use the identity $u_n(s)=f$ for all $s\in [0,h_n)$ 
 and Theorem~\ref{thm:time} to obtain
 \begin{equation} \label{eq:lb2}
  \sup_{s\in [0,h_n)}\|u_n(s)-u(s)\|_\kappa
  \leq e^{\omega h_n}\big(2r\delta_1+\rho_1\big(r, b_{0,1}r\delta_1^{-1}\big)h_n\big)
 \end{equation}
 for all $n\in\N$ and $s\in [0,h_n)$. Theorem~\ref{thm:cher}(vi) and Theorem~\ref{thm:time} imply
 \begin{align*}
  &|u(t,x)-u^{\epsilon}(t,x)|\kappa(x) \\
  &\leq\int_{\R_+\times\R^d}|u(t,x)-u(s+t, x+y)|\kappa(x)\eta^{\epsilon}(s,y)\,\d s\,\d y \\
  &\leq\int_{\R_+\times\R^d}|u(t,x)-u(s+t, x)|\kappa(x)\eta^{\epsilon}(s,y)\,\d s\, \d y \\
  &\quad\; +\int_{\R_+\times\R^d}|u(s+t, x)-u(s+t, x+y)|\kappa(x)\eta^{\epsilon}(s,y)\,\d s\,\d y \\
  &\leq e^{\omega(t+\epsilon_1)}\big(2r\delta_2+\rho_1\big(r, b_{0,1}r\delta_2^{-1}\big)\epsilon_1
  	+r\epsilon_2\big)
 \end{align*}
 for all $t\geq 0$ and $x\in\R^d$. Hence, for every $t\geq 0$,
 \begin{equation} \label{eq:lb3}
  \|u(t)-u^\epsilon(t)\|_\kappa
  \leq e^{\omega(t+\epsilon_1)}\big(2r\delta_2+\rho_1\big(r, b_{0,1}r\delta_2^{-1}\big)\epsilon_1
  	+r\epsilon_2\big).
 \end{equation}
 For every $n\in\N$ and $t\geq h_n$, combining the inequalities~\eqref{eq:lb1}--\eqref{eq:lb3} yields
 \begin{align*}
  \|(u_n(t)-u(t))^+\|_\kappa
  &\leq\|(u_n(t)-u^\epsilon(t))^+\|_\kappa+\|u^\epsilon(t)-u(t)\|_\kappa \\
  &\leq e^{\omega (t+h_n)}\big(2r\delta_1+\rho_1\big(r, b_{0,1}r \delta_1^{-1}\big)h_n\big) \\
  &\quad\; +e^{\omega (t+\epsilon_1)}(1+e^{\omega h_n})
   \big(2r\delta_2+\rho_1\big(r, b_{0,1}r\delta_2^{-1}\big)\epsilon_1+r\epsilon_2\big) \\
  &\quad\; +e^{\omega t}\big(Lre^{\omega(t+\epsilon_1)}\epsilon_2+\rho_2(\epsilon_1, \epsilon_2, h_n, r,t)\big)t.
 \end{align*}
 Moreover, for every $t\in [0,h_n]$, the claim follows directly from inequality \eqref{eq:lb2}.
\end{proof}

\subsection{Upper bound}

\begin{proof}[Proof of Theorem~\ref{thm:rate}(ii)]
 Let $r\geq 0$, $f\in\Lipb(r)$, $n\in\N$ and define $u_n(t):=I(\pi_n^t)f$ for all $t\geq 0$. 
 We fix $\delta=(\delta_1,\delta_2)$ with $\delta_1, \delta_2>0$ and $\epsilon=(\epsilon_1,\epsilon_2)$ 
 with $\epsilon_1, \epsilon_2\in (0, \epsilon_0]$. 
 In the sequel, we show that
 \begin{equation} \label{eq:ub1}
  S(t)f\leq v(t) \quad\mbox{for all } t\geq 0,
 \end{equation}
 where the function
 $v\colon\R_+\to\Ck,\; t\mapsto u_n^\epsilon(t)+e^{\omega t}\big(\frac{a+tb_t}{\kappa}\big)$
 depends on
 \begin{align*}
  a &:= e^{\omega h_n}\big(2r\delta_1+\rho_1\big(r, b_{0,1}r\delta_1^{-1}\big)h_n\big) +e^{\omega(h_n+\epsilon_1)}
  \big(2r\delta_2+\rho_1\big(r,b_{0,1}r\delta_2^{-1}\big)\epsilon_1+r\epsilon_2\big), \\
  b_t&:=Lre^{\omega\epsilon_1}\epsilon_2+\rho_3(\epsilon_1, \epsilon_2, h_n, r,t).
 \end{align*}
 For every $t\geq 0$ and $x\in\R^d$, Theorem~\ref{thm:time} implies
 \begin{align*}
  |u_n(t,x)-u_n^\epsilon(t,x)|\kappa(x) 
  &\leq\int_{\R_+\times\R^d}|u_n(t,x)-u_n(s+t, x+y)|\kappa(x)\eta^\epsilon(s,y)\,\d s\,\d y \\
  &\leq\int_{\R_+\times\R^d}|u_n(t,x)-u_n(s+t, x)|\kappa(x)\eta^\epsilon(s,y)\,\d s\, \d y \\
  &\quad\; +\int_{\R_+\times\R^d}|u_n(s+t, x)-u_n(s+t, x+y)|\kappa(x)\eta^\epsilon(s,y)\,\d s\,\d y \\
  &\leq e^{\omega(t+\epsilon_1)}\big(2r\delta_2+\rho_1\big(r, b_{0,1}r\delta_2^{-1}\big)\epsilon_1
  	+r\epsilon_2\big).
 \end{align*}
 Hence, for every $t\geq 0$,
 \begin{equation} \label{eq:ub2}
  \|u_n(t)-u_n^\epsilon(t)\|_\kappa
  \leq e^{\omega(t+\epsilon_1)}\big(2r\delta_2+\rho_1\big(r, b_{0,1}r\delta_2^{-1}\big)\epsilon_1
  	+r\epsilon_2\big).
 \end{equation}
 We apply Theorem~\ref{thm:time} and use $u_n(t)=f$ to obtain
 \begin{align*}
  \|S(t)f-u_n^\epsilon(t)\|_\kappa
  &\leq\|S(t)f-f\|_\kappa+\|f-u_n(t)\|_\kappa+\|u_n(t)-u_n^\epsilon(t)\|_\kappa \\
  &\leq e^{\omega h_n}\big(2r\delta_1+\rho_1\big(r,b_{0,1}r \delta_1^{-1}\big)h_n\big) \\
  &\quad\; +e^{\omega(h_n+\epsilon_1)}\big(2r\delta_2+\rho_1\big(r, b_{0,1}r\delta_2^{-1}\big)\epsilon_1
  +r\epsilon_2\big)
 \end{align*}
 and thus $S(t)f\leq v(t)$ for all $t\in [0, h_n)$. In addition, the continuity of the
 mappings $t\mapsto S(t)f$ and $t\mapsto u_n^\epsilon(t)$ and the monotonicity
 of the mapping $t\mapsto\rho_3(\epsilon_1, \epsilon_2, h_n, r,t)$ yields $S(h_n)f\leq v(h_n)$.
 In the sequel, we fix $T\geq h_n$ and show that
 \[ \tilde{v}\colon [h_n, T]\to\Ck,\;
 	t\mapsto u_n^\epsilon(t)+e^{\omega t}\Big(\frac{a+tb_T}{\kappa}\Big) \]
 satisfies the conditions from Theorem~\ref{thm:comp}. It holds $S(h_n)f\leq v(h_n)\leq\tilde{v}(h_n)$
 since the function $t\mapsto\rho_3(\epsilon_1, \epsilon_2,h_n,r,t)$ is non-decreasing.
 For every $t\in [h_n,T]$, we use $\Cbi\subset\L^S_+$ and Corollary~\ref{cor:kappa} to choose
 $c\geq 0$ and $h_0>0$ with
 \[ \|(S(h)\tilde{v}(t)-\tilde{v}(t))^+\|_\kappa
 	\leq\|(S(h)u_n^\epsilon(t)-u_n^\epsilon(t))^+\|_\kappa+(e^{\omega h}-1)e^{\omega t}(a+tb_T) \leq ch \]
 for all $h\in [0,h_0]$ which shows that $\tilde{v}(t)\in\L^S_+$. Furthermore, equation~\eqref{eq:kappa}
 and Theorem~\eqref{thm:cher}(ii) guarantee that
 \[ \|\tilde{v}(t)\|_\kappa\leq e^{\omega (T+\epsilon_1)}c_\kappa\|f\|_\kappa+e^{\omega T}(a+Tb_T)
 	\quad\mbox{for all } t\in [h_n, T]. \]
 Since $u_n^\epsilon|_{[h_n, T]}\in\Cbi([h_n,T]\times\R^d)$, there exists $c\geq 0$ with
 \[ \|\tilde{v}(s)-\tilde{v}(t)\|_\kappa
 	\leq c|s-t| \quad\mbox{for all } s,t\in [h_n,T]. \]
 This shows that the function $\tilde{v}\colon [h_n,T]\to\Ck$ is Lipschitz continuous and therefore satisfies
 inequality~\eqref{eq:comp1}. It remains to verify inequality~\eqref{eq:comp2}. For every $t\in [h_n,T]$
 and $h>0$, Corollary~\ref{cor:kappa} implies
 \[ \frac{S(h)\tilde{v}(t)-\tilde{v}(t)}{h}
 	\leq\frac{S(h)u_n^\epsilon(t)-u_n^\epsilon(t)}{h}
 	+\Big(\frac{e^{\omega h}-1}{h}\Big)\frac{e^{\omega t}(a+tb_T)}{\kappa}. \]
 Since $u_n^\epsilon(t)\in\Cbi\subset D(A)$, the right-hand side converges to
 \[ Au_n^\epsilon(t)+\omega e^{\omega t}\Big(\frac{a+tb_T}{\kappa}\Big) \quad\mbox{as } h\to 0. \]
 Moreover, for every $t\in [h_n, T]$, we use Assumption~\ref{ass:rate}(iii), Lemma~\ref{lem:super2},
 the equation
 \[ \partial_t \tilde{v}(t)
 	=\partial_t u_n^\epsilon(t)+\omega e^{\omega t}\Big(\frac{a+tb_T}{\kappa}\Big)
 	+\frac{e^{\omega t}b_T}{\kappa} \]
 and the monotonicity of the function $t\mapsto\rho_3(\epsilon_1,\epsilon_2,h_n, r,t)$ to obtain
 \begin{align*}
  &Au_n^\epsilon(t)+\omega e^{\omega t}\Big(\frac{a+tb_T}{\kappa}\Big) \\
  &=\partial_t u_n^\epsilon(t)+\omega e^{\omega t}\Big(\frac{a+tb_T}{\kappa}\Big)
   +Au_n^\epsilon(t)-\partial_t u_n^\epsilon(t) \\
  &\leq\partial_t \tilde{v}(t)-\frac{e^{\omega t}b_T}{\kappa}
   -\frac{u_n^\epsilon(t)-I(h_n)u_n^\epsilon(t-h_n)}{h_n}
   +\frac{\rho_3(\epsilon_1,\epsilon_2,h_n, r,t)}{\kappa} \\
  &\leq\partial_t \tilde{v}(t)-\frac{e^{\omega t}b_T}{\kappa}
   +\frac{Lre^{\omega (t+\epsilon_1)}\epsilon_2+\rho_3(\epsilon_1,\epsilon_2,h_n, r,t)}{\kappa}
   \leq\partial_t \tilde{v}(t).
 \end{align*}
 This shows that $\partial_t \tilde{v}(t)\geq A\tilde{v}(t)$ for all $t\in [h_n, T]$ and thus
 inequality~\eqref{eq:comp2} is satisfied. Theorem~\ref{thm:comp} yields $S(T)f\leq\tilde{v}(T)=v(T)$
 showing that inequality~\eqref{eq:ub1} is valid. Hence, for every $t\geq 0$, it follows from
 inequality~\eqref{eq:ub2} that
 \begin{align*}
  \|(S(t)f-u_n(t))^+\|_\kappa &\leq\|(S(t)f-v(t))^+\|_\kappa+\|v(t)-u_n(t)\|_\kappa \\
  &\leq\|u_n^\epsilon(t)-u_n(t)\|_\kappa+e^{\omega t}(a+tb_t) \\
  &\leq e^{\omega (t+h_n)}\big(2r\delta_1+\rho_1\big(r, b_{0,1}r\delta_1^{-1}\big)h_n\big) \\
  &\quad\; + e^{\omega(t+\epsilon_1)}\big(1+e^{\omega h_n}\big)
   \big(2r\delta_2+\rho_1\big(r, b_{0,1}r\delta_2^{-1}\big)\epsilon_1+r\epsilon_2\big) \\
  &\quad\; +e^{\omega t}\big(Lre^{\omega\epsilon_1}\epsilon_2+\rho_3(\epsilon_1, \epsilon_2, h_n, r,t)\big)t.
  \qedhere
 \end{align*}
\end{proof}

\subsection{Proof of Theorem~\ref{thm:rate2}}

\begin{proof}[Proof of Theorem~\ref{thm:rate2}]
 Let $r\geq 1$, $n\in\N$ and $h:=h_n$. First, by choosing $\delta_1:=h^{\gamma_1}$ 
 with $\gamma_1\geq 0$, Assumption~\ref{ass:rate2}(i)	implies
 \[ 2r\delta_1+\rho_1\big(r, b_{0,1}r\delta_1^{-1}\big)h
 	\leq 2rh^{\gamma_1}+\big(a_1(r)+a_2 b_{0,1}^p r^p h^{-p\gamma_1}\big)h. \]
 In order to minimize the right-hand side, we want to maximize $\gamma_1$ under the constraint
 $\gamma_1\leq 1-(i-1)p\gamma_1$ for $i=1,2$. Hence, we choose $ \gamma_1:=\frac{1}{1+p}$ 
 and obtain
 \begin{equation} \label{eq:delta1}
  2r\delta_1+\rho_1\big(r, b_{0,1}r\delta_1^{-1}\big)h
  \leq\big(2r+a_1(r)+a_2b_{0,1}^p r^p\big)h^\frac{1}{1+p}.
 \end{equation}
 Second, by choosing $\delta_2:=\epsilon_2$,  
 it follows from Assumption~\ref{ass:rate2}(i) and $\epsilon_1=\epsilon_2^{1+p}$ that 
 \begin{align}
  2r\delta_2+\rho_1\big(r, b_{0,1}r\delta_2^{-1}\big)\epsilon_1+r\epsilon_2 
  &\leq 2r\delta_2+\big(a_1(r)+a_2b_{0,1}^pr^p\delta_2^{-p}\big)\epsilon_2^{1+p}+r\epsilon_2 \nonumber \\
  &\leq\big(3r+a_1(r)+a_2b_{0,1}^pr^{p} \big)\epsilon_2 \label{eq:delta2}
 \end{align}
 Third, we choose $\epsilon_2:=h^{\gamma_2}\in (0,\epsilon_0]$ with $\gamma_2\geq 0$ 
 and use Assumption~\ref{ass:rate2}(ii) to obtain
 \[ Lre^{\omega(t+\epsilon_1)}\epsilon_2+\rho_2(\epsilon_1,\epsilon_2,h,r,t)
 	\leq Lre^{\omega(t+\epsilon_1)}h^{\gamma_2}
 	+\sum_{i=1}^{N^-}\theta_i^-(r,t)h^{\alpha_i^-}h^{-\gamma_2\beta_i^-}. \]
 In order to minimize the right-hand side, we want to maximize $\gamma_2$ under the constraint
 $\gamma_2\leq\alpha_i^--\gamma_2\beta_i^-$ for $i=1,\ldots,N^-$. Hence, we choose 
 \[ \gamma_2:=\min\bigg\{\frac{\alpha_1^-}{1+\beta_1^-}, \ldots,\frac{\alpha_{N^-}^-}{1+\beta_{N^-}^-}\bigg\} \] 
 and obtain
 \begin{equation} \label{eq:epsilon2}
  Lre^{\omega(t+\epsilon_1)}\epsilon_2+\rho_2(\epsilon_1,\epsilon_2,h,r,t)
  \leq \Big(Lre^{\omega(t+\epsilon_1)}+\sum_{i=1}^{N^-}\theta^-_i(r,t)\Big)h^{\gamma_2}.
 \end{equation}
 Combing inequalities~\eqref{eq:delta1}--\eqref{eq:epsilon2} leads to the choice 
 $\delta_1=\delta_2=\epsilon_2=h^{\gamma^-}$ with
 \[ \gamma^-:=\min\bigg\{\frac{1}{1+p},\frac{\alpha_1^-}{1+\beta_1^-}, \ldots,\frac{\alpha_{N^-}^-}{1+\beta_{N^-}^-}\bigg\} \]
 and $\epsilon_1=\epsilon_2^{1+p}=h^{(1+p)\gamma^-}\geq h$.
 Hence, for every $r\geq 1$, $t\geq 0$, $f\in\Lipb(r)$ and $n\in\N$, 
 it follows from Theorem~\ref{thm:rate}(i) that
 \[ \|(S(t)f-I(\pi_n^t)f)^-\|_\kappa\leq c_{r,t}^-h_n^{\gamma^-}, \]
 where $c_{r,t}^-$ is given as in the statement.	Similarly, one can prove the upper bound.
\end{proof}

\subsection{Proof of Theorem~\ref{thm:rate3}}

We briefly explain how the result for bounded functions can be transferred
to arbitrary Lipschitz continuous functions. Apart from modifying the assumptions,
the arguments do not change.

\begin{proof}[Proof of Theorem~\ref{thm:rate3}]
 Assumption~\ref{ass:rate3}(iv) ensures that Theorem~\ref{thm:time} remains valid
 with $\Lipb(r)$ instead $\Lip(r)$. Regarding Theorem~\ref{thm:cher}(v), the same follows
 from Assumption~\ref{ass:rate3}(i) and~(iii) and thus Lemma~\ref{lem:super} and
 Lemma~\ref{lem:super2} can be applied for arbitrary Lipschitz continuous functions
 as well.\ We now focus on the proof of Theorem~\ref{thm:rate}(i). Since the mapping 
 $t\mapsto u^\epsilon(t)$ is differentiable with
 \begin{equation} \label{eq:unbounded}
  \|\partial_t u^\epsilon(t)\|_\kappa
  \leq c_\kappa\|u(t)\|_\kappa\sup_{(s,x)\in [0,\epsilon_1]\times B(\epsilon_2)}|\partial_s \eta^\epsilon(s,x)|
  \quad\mbox{for all } t\geq 0
 \end{equation}
 and due to Assumption~\ref{ass:rate3}(v), we can apply Lemma~\ref{lem:super}. 
 The rest of the proof does not change because Theorem~\ref{thm:time},
 Lemma~\ref{lem:BJ} and Theorem~\ref{thm:cher}(v) are valid for unbounded functions.\
 Regarding the proof of Theorem~\ref{thm:rate}(ii), equation~\eqref{eq:unbounded} guarantees
 that~$\tilde{v}$ is locally Lipschitz continuous in time. Since $u_n^\epsilon(t)\in\Lip\cap\Ci$ 
 has bounded partial derivatives of any order, Assumption~\ref{ass:rate}(ii) and 
 Theorem~\ref{thm:cher}(i) imply $u_n^\epsilon(t)\in D(A)$. Furthermore, it follows inductively 
 from Assumption~\ref{ass:rate}(iii) that $u_n(t)\in\Lip(e^{\omega t})$ and therefore 
 $u_n^\epsilon(t)\in\Lip(e^{\omega (t+\epsilon_1)})$ which ensures that Lemma~\ref{lem:super2} 
 can be applied.
\end{proof}

\section{Examples}
\label{sec:examples}

\subsection{Nisio semigroups}

Let $(S_\lambda)_{\lambda\in\Lambda}$ be a family of strongly continuous linear semigroups 
$(S_\lambda(t))_{t\geq 0}$ on $\Ck$ with generators $(A_\lambda)_{\lambda\in\Lambda}$.\ 
The following conditions can be verified for suitably bounded families of L\'evy processes, 
Ornstein--Uhlenbeck processes and geometric Brownian motions, see~\cite{NR21, DKN20, BK22+}.
This remains valid if one replaces $\Lipb$ by $\Lip$ to be in the framework of Subsection~\ref{sec:unbounded}.

\begin{assumption} \label{ass:nisio}
 Suppose that the conditions~(i)--(v) and either (vi) or (vi') from the following list are satisfied.
 \begin{itemize}
  \item[(i)] There exists $\omega\geq 0$ with $\|S_\lambda(t)f\|_\kappa\leq e^{\omega t}\|f\|_\kappa$
   for all $t\geq 0$, $f\in\Ck$ and $\lambda\in\Lambda$.
  \item[(ii)] There exist $t_0>0$, $\epsilon_0\in (0,1]$ and $L\geq 0$ with
   \[ \|S_\lambda(t)(\tau_x f)-\tau_x S_\lambda(t)f\|_\kappa\leq Lrt|x| \]
   for all $t\in [0,t_0]$, $x\in B_{\R^d}(\epsilon_0)$, $r\geq 0$, $f\in\Lipb(r)$ and
   $\lambda\in\Lambda$.
  \item[(iii)] It holds $S_\lambda(t)\colon\Lipb(r)\to\Lipb(e^{\omega t}r)$ for all $r,t\geq 0$
   and $\lambda\in\Lambda$.
  \item[(iv)] It holds $\Cbi\subset\bigcap_{\lambda\in\Lambda}D(A_\lambda)$.
   For every $f\in\Cbi$ and $K\Subset\R^d$,
   \[ \lim_{h\downarrow 0}\sup_{\lambda\in\Lambda}
   	\left\|\frac{S_\lambda(h)f-f}{h}-A_\lambda f\right\|_{\infty, K}=0. \]
  \item[(v)] There exist $v_1,v_2,w_1,w_2,w_3\ge 0$ with
   \[ \|A_\lambda f\|_\kappa\leq\sum_{i=1}^2 v_i d^{-\frac{i}{2}}\|D^i f\|_\infty
   	\quad\mbox{and}\quad
   	A_\lambda f\in\Lipb\bigg(\sum_{i=1}^{3} w_i d^{-\frac{i}{2}}\|D^i f\|_\infty\bigg) \]
   for all $f\in\Cbi$ and $\lambda\in\Lambda$.
  \item[(vi)] There exists a bounded continuous function $\tilde{\kappa}\colon\R^d\to (0,\infty)$
   such that, for every $\epsilon>0$, there exists $K\Subset\R^d$ with
   $\sup_{x\in K^c}\frac{\tilde{\kappa}(x)}{\kappa(x)}\leq\epsilon$. Moreover,
   for every $T\geq 0$, there exists $c\geq 0$ with
   $\|S_\lambda(t)f\|_{\tilde{\kappa}}\leq c\|f\|_{\tilde{\kappa}}$
   for all $t\in [0,T]$, $f\in\Ck$ and $\lambda\in\Lambda$.
  \item[(vi')] For every $\epsilon>0$ and $K\Subset\R^d$, there exist a family
   $(\zeta_x)_{x\in K}$ of continuous functions $\zeta_x\colon\R^d\to\R$
   and $\tilde{K}\Subset\R^d$ with
   \begin{enumerate}[label=(\alph*)]
    \item $0\leq\zeta_x\leq 1$ and $\zeta_x(x)=1$,
    \item $\sup_{y\in\tilde{K}}\zeta_x(y)\leq\epsilon$ for all $x\in K$,
    \item The function $f:=\frac{r}{\kappa}(1-\zeta_x)$ satisfies $f\in D(A_\lambda)$,
     $\|A_\lambda f\|_\kappa\leq\epsilon$ and
     \[ \lim_{h\downarrow 0}\left\|\frac{S_\lambda(h)f-f}{h}-A_\lambda f\right\|_\kappa=0
     	\quad\mbox{for all } x\in K \mbox{ and } \lambda\in\Lambda. \]
   \end{enumerate}
 \end{itemize}
\end{assumption}

The conditions~(vi) and~(vi') are two sufficient criteria from~\cite{BKN23} in order to 
guarantee that Assumption~\ref{ass:cher}(vi) is satisfied.\ The first one is a moment
condition while the second one characterizes continuity from above by a suitable
family of cut-off functions. For every $t\geq 0$, $f\in\Ck$ and $x\in\R^d$, we define
\[ (I(t)f)(x):=\sup_{\lambda\in\Lambda}\,(S_\lambda(t)f)(x). \]

\begin{theorem} \label{thm:nisio}
 Suppose that Assumption~\ref{ass:nisio} is satisfied. Then, there exists a strongly continuous
 convex monotone semigroup $(S(t))_{t\geq 0}$ on $\Ck$ with $S(t)0=0$ given by
 \[ S(t)f=\lim_{n\to\infty}I(\pi_n^t)f \quad\mbox{for all } (f,t)\in\Ck\times\R_+ \]
 such that $\Cbi\subset D(A)$ and $Af=\sup_{\lambda\in\Lambda}A_\lambda f$
 for all $f\in\Cbi$. Furthermore,
 \[ 0\leq\big(S(t)f-I(\pi_n^t)f\big)\kappa\leq c_{r,t}^+h_n^{\gamma^+} \]
 for all $r,t\geq 0$, $f\in\Lipb(r)$ and $n\in\N$ with $h_n^{\gamma^+}\leq\epsilon_0$, where 
 \[ \gamma^+:=\begin{cases}
 	\nicefrac{1}{2}, & v_2=w_3=0, \\
 	\nicefrac{1}{6}, & \mbox{otherwise}. \end{cases} \]
 and $c_{r,t}^+$ is given by equation~\eqref{eq:nisio.crt}.
\end{theorem}
\begin{proof}
 First, we verify Assumption~\ref{ass:cher}. For every $t\geq 0$, $f,g\in\Ck$ and $x,y\in\R^d$,
 \begin{align*}
  \|I(t)f-I(t)g\|_\kappa &\leq\sup_{\lambda\in\Lambda}\|S_\lambda(t)f-S_\lambda(t)g\|_\kappa, \\
  |(I(t)f)(x)-(I(t)f)(y)| &\leq\sup_{\lambda\in\Lambda}|(S_\lambda(t)f)(x)-(S_\lambda(t)f)(y)|
 \end{align*}
 and therefore Assumption~\ref{ass:nisio} yields that the conditions~(i)--(iv) and~(vii) are satisfied.
 Since Assumption~\ref{ass:nisio}(i) and~(v) imply
 \[ \sup_{\lambda\in\Lambda}\|S_\lambda(t)f-f\|_\kappa
 	\leq\sup_{\lambda\in\Lambda}\int_0^t \|S_\lambda(s)A_\lambda f\|_\kappa\,\d s
 	\leq e^{\omega t}t\sup_{\lambda\in\Lambda}\|A_\lambda f\|_\kappa<\infty \]
 for all $t\geq 0$ and $f\in\Cbi$, we obtain from Assumption~\ref{ass:nisio}(iv) that 
 Assumption~\ref{ass:cher}(v) is valid.
 Furthermore, Assumption~\ref{ass:nisio}(vi) and~(vi') together with~\cite[Lemma~2.11 and Corollary~2.15]{BKN23}
 guarantee that condition~(vi) is satisfied.\
 The first part of the claim follows from Theorem~\ref{thm:cher} and the inequality
 \[ I(\pi_n^t)f\leq S(t)f \quad\mbox{for all } t\geq 0, f\in\Ck \mbox{ and } n\in\N \]
 holds by construction. Indeed, Theorem~\ref{thm:comp} implies 
 $I(t)f=\sup_{\lambda\in\Lambda}S_\lambda(t)f\leq S(t)f$ and
 the semigroup property guarantees $I(s)I(t)f\leq S(s+t)f$ for all $s,t\geq 0$ and $f\in\Ck$. 

 Second, we verify Assumption~\ref{ass:rate}(i) and~(iii). For every $\lambda\in\Lambda$,
 $h\in [0,1]$ and $f\in\Cbi$, it follows from Assumption~\ref{ass:nisio}(v) that
 \[ \|S_\lambda(h)f-f\|_\kappa
 	\leq e^{\omega h}h\sup_{\lambda\in\Lambda}\|A_\lambda f\|_\kappa
 	\leq e^\omega\big(v_1 d^{-\frac{1}{2}}\|D^1 f\|_\infty+v_2 d^{-1}\|D^2 f\|_\infty\big)h. \]
 Taking the supremum over $\lambda\in\Lambda$ yields that Assumption~\ref{ass:rate}(i)
 is satisfied with 
 \[ \rho_1(x_1, x_2):=e^\omega(v_1 x_1+v_2 x_2)=a_1(x_1)+a_2 x_2^p \quad\mbox{for all } x_1, x_2\geq 0, \]
 where $a_1(x_1):=e^\omega v_1 x_1$, $a_2:=e^\omega v_2$ and 
 \[ p:=\begin{cases}
 	0, & a_2=0, \\
 	1, & a_2>0.
 \end{cases} \]
 For every $r,T\geq 0$, $s,t\in [0,T]$ and $f\in\Lipb(r)$, Corollary~\ref{cor:time} implies
 \begin{equation} \label{eq:nisio.time}
  \|S_\lambda(s)f-S_\lambda(t)f\|_\kappa\leq c_{r,T}|s-t|^\alpha,
 \end{equation}
 where $c_{r,T}:=e^{\omega T}\big(2+e^\omega v_1+e^\omega v_2 b_{0,1}\big)r$ and 
 $\alpha:=\frac{1}{1+p}$. Let $\epsilon=(\epsilon_1,\epsilon_2)$ with $\epsilon_1,\epsilon_2\in (0,\epsilon_0]$, 
 $n\in\N$, $r\geq 0$, $f\in\Lipb(r)$, $t\geq h_n$ and  $u_n(t):=I(\pi_n^t)f$. It holds
 \begin{align}
  &\partial_t u_n^\epsilon(t)-Au_n^\epsilon(t)-\frac{u_n^\epsilon(t)-I(h_n)u_n^\epsilon(t-h_n)}{h_n} \nonumber \\
  &=\frac{I(h_n)u_n^\epsilon(t-h_n)-u_n^\epsilon(t-h_n)}{h_n}-Au_n^\epsilon(t-h_n) 
 +Au_n^\epsilon(t-h_n)-Au_n^\epsilon(t) \nonumber \\
  &\quad\; +\partial_t u_n^\epsilon(t)-\frac{u_n^\epsilon(t)-u_n^\epsilon(t-h_n)}{h_n}. \label{eq:nisio.ub}
 \end{align}
 Assumption~\ref{ass:nisio}(v), inequality~\eqref{eq:nisio.time} and Lemma~\ref{lem:deriv} imply
 \begin{align*}
  &\left\|\frac{I(h_n)u_n^\epsilon(t-h_n)-u_n^\epsilon(t-h_n)}{h}-Au_n^\epsilon(t-h_n)\right\|_\kappa \\
  &\leq\sup_{\lambda\in\Lambda}
  	\left\|\frac{S_\lambda(h_n)u_n^\epsilon(t-h_n)-u_n^\epsilon(t-h_n)}{h_n}-A_\lambda u_n^\epsilon(t-h_n)\right\|_\kappa \\
  &\leq\sup_{\lambda\in\Lambda}\frac{1}{h_n}\int_0^{h_n}
  	\|S_\lambda(s)A_\lambda u_n^\epsilon(t-h_n)-A_\lambda u_n^\epsilon(t-h_n)\|_\kappa\,\d s \\
  &\leq\sup_{\lambda\in\Lambda}\frac{1}{h_n}\int_0^{h_n}
  	c_{r,t} e^{\omega(t-h_n+\epsilon_1)}
  	\sum_{i=1}^{3}w_ib_{0,i-1}\epsilon_2^{1-i}s^\alpha\,\d s \\
  &\leq\frac{c_{r,t} }{1+\alpha}e^{\omega(t-h_n+\epsilon_1)}
  	\sum_{i=1}^{3}w_ib_{0,i-1} h_n^\alpha\epsilon_2^{1-i}.
 \end{align*}
 Moreover, it follows from Assumption~\ref{ass:nisio}(v) and Lemma~\ref{lem:deriv2} that
 \begin{align*}
  &\|Au_n^\epsilon(t-h_n)-Au_n^\epsilon(t)\|_\kappa
  \leq\sup_{\lambda\in\Lambda}\|A_\lambda u_n^\epsilon(t-h_n)-A_\lambda u_n^\epsilon(t)\|_\kappa \\
  &\leq\sum_{i=1}^2 v_i d^{-\frac{i}{2}}\|D^i u_n^\epsilon(t-h_n)-D^i u_n^\epsilon(t)\|_\kappa 
  \leq\sum_{i=1}^2 v_i d^{-\frac{i}{2}} \int_0^{h_n} \|\partial_t D^i u	^\epsilon(t-s)\|_\kappa\,\d s \\
  &\leq\big(c_\kappa c_{r,t}(\epsilon_1+h_n)^\alpha+e^{\omega t}r\epsilon_2\big)
  	\sum_{i=1}^2 v_i b_{1,i} h_n\epsilon_1^{-1}\epsilon_2^{-i}
 \end{align*}
 and
 \[ \left\|\partial_t u_n^\epsilon(t)-\frac{u_n^\epsilon(t)-u_n^\epsilon(t-h_n)}{h_n}\right\|_\kappa
 	\leq\frac{1}{2}\big(c_\kappa c_{r,t}(\epsilon_1+h_n)^\alpha+e^{\omega t}r\epsilon_2)
 	b_{2,0} h_n\epsilon_1^{-2}. \]
 We obtain that Assumption~\ref{ass:rate}(iii) is satisfied with
 \begin{align*}
  \rho_3(\epsilon_1,\epsilon_2,h_n,r,t)
  &:=\frac{c_{r,t} }{1+\alpha}e^{\omega(t-h_n+\epsilon_1)}
  	\sum_{i=1}^{3}w_ib_{0,i-1} h_n^\alpha\epsilon_2^{1-i} \\
  &\quad\; +\big(c_\kappa c_{r,t}(\epsilon_1+h_n)^\alpha+e^{\omega t}r\epsilon_2\big)
  	\sum_{i=1}^2 v_i b_{1,i} h_n\epsilon_1^{-1}\epsilon_2^{-i} \\
  &\quad\; +\frac{1}{2}\big(c_\kappa c_{r,t}(\epsilon_1+h_n)^\alpha+e^{\omega t}r\epsilon_2) b_{2,0} h_n\epsilon_1^{-2}.
 \end{align*}
 In particular, for $\epsilon_1=\epsilon_2^{1+p}\geq h_n$, we obtain 
 that Assumption~\ref{ass:rate2}(iii) is satisfied. Hence, Theorem~\ref{thm:rate2}(ii) implies
 \[ 0\leq\big(S(t)f-I(\pi_n^t)f\big)\kappa\leq  c_{r,t}^+h_n^{\gamma^+} \]
 for all $r,t\geq 0$, $f\in\Lipb(r)$ and $n\in\N$ with $h_n^{\gamma^+}\leq\epsilon_0$, where
 \[ \gamma^+:=\begin{cases}
 	\nicefrac{1}{2}, & v_2=w_3=0, \\
 	\nicefrac{1}{6}, & \mbox{otherwise} \end{cases} \]
 and, for $\epsilon_1^+:=h_0^{(1+p)\gamma^+}$ with $h_0:=\max_{n\in\N}h_n$, 
 \begin{align}
  c_{r,t}^+ &:=e^{\omega(t+h_0)}\big(2r+e^\omega v_1 r+e^\omega v_2b_{0,1}^p r^{p}\big) \label{eq:nisio.crt}  \\
  &+\big(e^{\omega(t+\epsilon_1^+)}\big(1+e^{\omega h_0}\big)\big)
  	\big(3r+e^\omega v_1 r+e^\omega v_2 b_{0,1}^p r^{p}\big) \nonumber\\
  &+e^{\omega t}\Big(Lre^{\omega \epsilon_1^+}
  	+\Big(\frac{c_{r,t}}{1+\alpha}e^{\omega(t+\epsilon_1^+)}\sum_{i=1}^{3}w_ib_{0,i-1} 
  	+(2c_\kappa c_{r,t}+e^{\omega t}r)\Big(\sum_{i=1}^2 v_i b_{1,i} +\frac{1}{2}b_{2,0}\Big)\Big)t. \nonumber
 \end{align}
 We do not require $r\geq 1$ since Assumption~\ref{ass:rate2}(i) is satisfied for all $x_1,x_2\geq 0$. 
\end{proof}

In the second order case, the rate $\gamma^+=\nicefrac{1}{6}$ is determined by the term 
$w_3b_{0,3}h_n^\alpha\epsilon_2^{-2}$. The remaining terms would lead to the better 
convergence rate of $\gamma^+=\nicefrac{1}{4}$ which can be achieved under an additional 
condition on the family $(A_\lambda)_{\lambda\in\Lambda}$ of linear generators. 
In applications this typically means that the coefficients of the linear generators have to be
sufficiently smooth.

\begin{theorem} \label{thm:nisio2}
In addition to the assumptions of Theorem~\ref{thm:nisio}, we suppose that there exist
 $\tilde{v}_1, \ldots,\tilde{v}_{4}\geq 0$ with
 \begin{equation}  \label{eq:nisio2}
  A_\lambda f\in D(A_\lambda) \quad\mbox{and}\quad
  \|A_\lambda^2 f\|_\kappa\leq\sum_{i=1}^{4}\tilde{v}_id^{-\frac{i}{2}}\|D^i f\|_\infty
 \end{equation}
 for all $f\in\Cbi$ and $\lambda\in\Lambda$.\ Then, for every $r,t\geq 0$, $f\in\Lipb(r)$ and $n\in\N$
 with $h_n^{\gamma^+}\leq\epsilon_0$, 
 \[ 0\leq\big(S(t)f-I(\pi_n^t)f\big)\leq c_{r,t}^+h_n^{\gamma^+}
 	\quad\mbox{with}\quad \gamma^+:=\frac{1}{4}, \]
 where $c_{r,t}^+$ is given by equation~\eqref{eq:nisio2.crt}.
\end{theorem}
\begin{proof}
 Again, we choose $\rho_1(x_1, x_2):=e^\omega(\tilde v_1 x_1+\tilde v_2 x_2)$.
 Let $\epsilon=(\epsilon_1,\epsilon_2)$ with $\epsilon_1,\epsilon_2\in (0,\epsilon_0]$,
 $n\in\N$, $r\geq 0$, $f\in\Lipb(r)$ and $t\geq h_n$.
 Assumption~\ref{ass:nisio}(i) and Lemma~\ref{lem:deriv} yield
 \begin{align*}
  &\left\|\frac{I(h_n)u_n^\epsilon(t-h_n)-u_n^\epsilon(t-h_n)}{h_n}-Au_n^\epsilon(t-h_n)\right\|_\kappa \\
  &\leq\sup_{\lambda\in\Lambda}\frac{1}{h_n}\int_0^{h_n}
  	\|S_\lambda(s_1)A_\lambda u_n^\epsilon(t-h_n)-A_\lambda u_n^\epsilon(t-h_n)\|_\kappa\,\d s_1 \\
  &\leq\sup_{\lambda\in\Lambda}\frac{1}{h_n}\int_0^{h_n}
  	\int_0^{s_1}\|S_\lambda(s_2)A_\lambda^2 u_n^\epsilon(t-h_n)\|_\kappa\,\d s_2\,\d s_1 \\
  &\leq\frac{1}{2}e^{\omega h_n}\sum_{i=1}^{4}\tilde{v}_i d^{\frac{i}{2}}\|D^i u_n^\epsilon(t-h_n)\|_\infty  h_n
  	\leq\frac{1}{2}e^{\omega(t+\epsilon_1)}r\sum_{i=1}^{4}\tilde{v}_ib_{0,i-1} h_n\epsilon_2^{1-i}.
 \end{align*}
 Combining the previous estimate with equation~\eqref{eq:nisio.ub} and the estimates
 from the proof of Theorem~\ref{thm:nisio} show that Assumption~\ref{ass:rate}(iii)
 is satisfied 
 for  $\epsilon_1=\epsilon_1^2\geq h_n$ and $\alpha=\nicefrac{1}{2}$.
 Hence, we can apply Theorem~\ref{thm:rate2}(ii) to obtain  
 \[ 0\leq\big(S(t)f-I(\pi_n^t)f\big)\kappa\leq  c_{r,t}^+h^{\gamma^+} \] 
 for all all $r,t\geq 0$, $f\in\Lipb(r)$ and $n\in\N$, where $\gamma^+:=\nicefrac{1}{4}$ and 
 \begin{align}
  c_{r,t}^+ &:=e^{\omega(t+h_0)}\big(2r+e^\omega v_1 r+e^\omega v_2 b_{0,1}^p r^{p}\big)\label{eq:nisio2.crt}  \\
  &\quad\;+\big(e^{\omega(t+\epsilon_1^+)}\big(1+e^{\omega h_0}\big)\big)
  	\big(3r+e^\omega v_1 r+e^\omega v_2 b_{0,1}^p r^{p}\big) \nonumber\\
  &\quad\; +e^{\omega t}\Big(Lre^{\omega \epsilon_1^+}
  	+\frac{1}{2}e^{\omega(t+\epsilon_1^+)}r\sum_{i=1}^{4}\tilde{v}_ib_{0,i-1} 
  	+(2c_\kappa c_{r,t}+e^{\omega t}r)\Big(\sum_{i=1}^2 v_i b_{1,i}+\frac{1}{2}b_{2,0}\Big)\Big)t \nonumber
 \end{align}
 with $\epsilon_1^+=h_0^{2\gamma^+}$ for $h_0:=\max_{n\in\N}h_n$.
\end{proof}

As seen during the proof of the previous theorem, the condition
\[ \lim_{h\downarrow 0}\sup_{\lambda\in\Lambda}
	\left\|\frac{S_\lambda(h)f-f}{h}-A_\lambda f\right\|_{\infty, K}=0 \]
from Assumption~\ref{ass:nisio} is guaranteed by equation~\eqref{eq:nisio2}.

\begin{example}
 Let $\Lambda\subset\S^d_+\times\R^d$ be a bounded set, where $\S^d_+$ contains 
 all symmetric positive semi-definite $d\times d$-matrices. For every $(\sigma,m)\in\Lambda$
 and $x\in\R^d$, we define 
 \[ (S_{\sigma,m}(t)f)(x):=\e[f(x+\sigma W_t+mt)], \]
 where $(W_t)_{t\geq 0}$ is a standard $d$-dimensional Brownian motion. 
 Due to It\^o's formula, the generator is given by 
 \[ (A_{\sigma,m}f)(x)=\frac{1}{2}\tr(\sigma^2 D^2f(x))+m^T Df(x)
 	\quad\mbox{for all } f\in\Cb^2 \mbox{ and } x\in\R^d. \]
 Let $I(t)f:=\sup_{(\sigma,m)\in\Lambda}S_{\sigma,m}(t)f$ for all $t\geq 0$ and $f\in\Cb$.
 Using It\^o's formula again, it is straightforward to show that Assumption~\ref{ass:nisio} 
 is satisfied, where condition~(vi') follows, for instance, from~\cite[Corollary~2.16]{BKN23}. 
 Moreover, by~\cite[Theorem~5.5]{BDKN22}, the corresponding semigroup can
 be represented as the value function of a dynamic stochastic control problem, i.e., 
 \[ (S(t)f)(x)=\sup_{(\sigma,m)\in\A}\e\left[f\left(x+\int_0^t \sigma_s\,\d W_s+\int_0^t m_s\,\d s\right)\right] \]
 for all $t\geq 0$, $f\in\Cb$ and $x\in\R^d$, where $\A$ contains all predictable
 processes taking values in $\Lambda$. Since the coefficients of $A_{\sigma,m}$
 are constant, Theorem~\ref{thm:nisio2} yields
 \[ 0\leq\big(S(t)f-I(\pi_n^t)f\big)\leq c_{r,t}^+h_n^{\gamma^+}
 	\quad\mbox{with}\quad \gamma^+:=\frac{1}{4} \]
 for all $r,t\geq 0$, $f\in\Lipb(r)$ and $n\in\N$, where $c_{r,t}$ depends on
 $\sup_{(\sigma,m)\in\Lambda}(|\sigma|^2+|m|)$. 
 
  The result can be extended to bounded Lipschitz continuous drifts
  $m\colon\R^d\to\R^d$ by defining $(S_{\sigma,m}(t)f)(x):=\e[f(X_t^{(\sigma,m,x)}]$,
  where $X^{(\sigma,m,x)}$ denotes the unique solution of 
  \[ \d X_t^{(\sigma,m,x)}=\sigma\,\d W_t+m(X_t^{(\sigma,m,x)})\,\d t,
 	\quad X_0^{(\sigma,m,x)}=x. \]
 The generator is now given by 
  \[ (A_{\sigma,m}f)(x)=\frac{1}{2}\tr(\sigma^2 D^2f(x))+m(x)^TDf(x)
 	\quad\mbox{for all } f\in\Cb^2 \mbox{ and } x\in\R^d. \]
 While Theorem~\ref{thm:nisio} can always be applied, Theorem~\ref{thm:nisio2} requires 
 $A_{\sigma,m}\in D(A_{\sigma,m})$ which is satisfied for $m\in\Cb^2$.
\end{example}

\subsection{Law of large numbers for convex expectations}

Let $\kappa\equiv 1$ and denote by $\Cb$ the space of all bounded continuous
functions $f\colon\R^d\to\R$ endowed with the supremum norm $\|\cdot\|_\infty$.
Let $(\xi_n)_{n\in\N}\subset\H^d$ be an iid sequence of random vectors
$\xi_n\colon\Omega\to\R^d$ on a convex expectation space $(\Omega,\H,\E)$
with $\lim_{c\to\infty}\E[(|\xi_1|-c)^+]=0$. In addition, we assume $f(X)\in\H$
for all $f\in\Cb$ and $X\in\H^d$. In~\cite[Theorem 3.4]{BK22}, it was shown that
\begin{equation} \label{eq:LLN.conv}
 \frac{1}{n}\E\left[nf\left(\frac{1}{n}\sum_{i=1}^n \xi_i\right)\right]\to\bar{\E}[f(\zeta)]
 \quad\mbox{for all } f\in\Cb,
\end{equation}
where $\zeta\colon\bar{\Omega}\to\R^d$ is maximally distributed, i.e.,
\[ \bar{\E}[f(\zeta)]=\sup_{y\in\R^d}\big(f(y)-\varphi(y))
	\quad\mbox{with}\quad \varphi(y):=\sup_{z\in\R^d}(yz-\E[z\xi_1]). \]
For an overview on convex expectations, we refer to~\cite{Peng19} and~\cite[Appendix~B]{BK22}.
W.l.o.g., we subsequently assume that the random vectors $(\xi_n)_{n\in\N}$
and $\zeta$ are defined on the same convex expectation space. The converge
in equation~\eqref{eq:LLN.conv} was obtained by applying Theorem~\ref{thm:cher}
with $I(0):=\id_{\Cb}$ and
\[ (I(t)f)(x):=t\E\left[\frac{1}{t}f(x+t\xi_1)\right] \]
for all $t>0$, $f\in\Cb$ and $x\in\R^d$. In addition, the generator of the corresponding
semigroup $(S(t))_{t\geq 0}$ is given by
\[ (Af)(x)=\E[Df(x)\xi_1] \quad\mbox{for all } f\in\Cb^1 \mbox{ and } x\in\R^d, \]
where $\Cb^1$ consists of all bounded continuously differentiable functions
$f\colon\R^d\to\R$ with bounded first order partial derivatives. Under an additional 
moment condition, we now provide a rate for the convergence in equation~\eqref{eq:LLN.conv}.

\begin{theorem} \label{thm:LLN}
 Let $\E[c|\xi_1|^2]<\infty$ for all $c\geq 0$. Then,
 \[ -c_{r,t}^-\sqrt{h_n}\leq\frac{1}{n}\E\left[nf\left(\frac{1}{n}\sum_{i=1}^n \xi_i\right)\right]-\E[f(\zeta)]
 	\leq c_{r,t}^+\sqrt{h_n} \]
 for all $r\geq 1$, $f\in\Lipb(r)$ and $n\in\N$, where $c_{r,t}^\pm$ are given by equation~\eqref{eq:LLN.cr-}
 and~\eqref{eq:LLN.cr+}.
\end{theorem}
\begin{proof}
 First, we verify Assumption~\ref{ass:rate}(i). For every $h>0$, $f\in\Cb^1$,
 and $x\in\R^d$,
 \begin{align*}
  \left|\frac{I(h)f-f}{h}\right|(x) &\leq\E\left[\frac{|f(x+h\xi_1)-f(x)|}{h}\right] \\
  &\leq\E\left[\frac{1}{h}\int_0^h |Df(x+h\xi_1)\xi_1|\,\d s\right]
  \leq\E[\|Df\|_\infty |\xi_1|].
 \end{align*}
 Hence, we can choose $\rho_1(x_1,x_2):=a_1(x_1):=\E[d^{\frac{1}{2}}x_1|\xi_1|]$ 
 and Corollary~\ref{cor:time} can be applied with $\alpha:=1$, $\omega:=0$ and 
 $c_r:=c_{r,T}:=2r+\E[d^{\frac{1}{2}}r|\xi_1|]$.

 Second, we verify Assumption~\ref{ass:rate}(ii) and~(iii). Let $\epsilon=(\epsilon_1,\epsilon_2)$ 
 with $\epsilon_1,\epsilon_2>0$, $n\in\N$, $r\geq 1$, $t\geq h_n$ and $f\in\Lipb(r)$. 
 With $h:=h_n$, it holds
 \begin{align*}
  \partial_t u^\epsilon(t)-Au^\epsilon(t)-\frac{u^\epsilon(t)-I(h)u^\epsilon(t-h)}{h}
  &=\frac{I(h)u^\epsilon(t-h)-u^\epsilon(t-h)}{h}-Au^\epsilon(t-h) \\
  &\quad\; +Au^\epsilon(t-h)-Au^\epsilon(t) \\
  &\quad\; +\partial_t u^\epsilon(t)-\frac{u^\epsilon(t)-u^\epsilon(t-h)}{h}.
 \end{align*}
 Applying Lemma~\ref{lem:lambda} with $\lambda:=h\epsilon_2^{-1}\in (0,1]$ and
 Lemma~\ref{lem:deriv} yields
 \begin{align*}
  &\left(\frac{I(h)u^\epsilon(t-h)-u^\epsilon(t-h)}{h}-Au^\epsilon(t-h)\right)(x) \\
  &=\E\left[\frac{1}{h}\int_0^h Du^\epsilon(t-h,x+s\xi_1)\xi_1\,\d s\right]
  	-\E[Du^\epsilon(t-h),x)\xi_1] \\
  &\leq\lambda\E\left[\frac{1}{\lambda h}\int_0^h |(Du^\epsilon(t-h,x+s\xi_1)-Du^\epsilon(t-h,x))\xi_1|\,\d s
  	+|Du^\epsilon(t-h,x)\xi_1|\right] \\
  &\quad\; +\lambda\E[|Du^\epsilon(t-h,x)\xi_1|] \\
  &\leq\lambda\E\left[\frac{1}{\lambda h}\int_0^h \|D^2 u^\epsilon(t-h)\|_\infty |\xi_1|^2 s\,\d s+\|Du^\epsilon(t-h)\|_\infty |\xi_1|\right]
  	+\lambda\E[\|Du^\epsilon(t-h)\|_\infty |\xi_1|] \\
  &\leq\lambda\E\left[\frac{h\epsilon_2^{-1}}{2\lambda}dr b_{0,1}|\xi_1|^2+d^\frac{1}{2}r|\xi_1|\right]
  	+\lambda\E[d^\frac{1}{2}r|\xi_1|] \\
  &=\left(\E\Big[\frac{1}{2}drb_{0,1}|\xi_1|^2+d^\frac{1}{2}r|\xi_1|\Big]+\E[d^\frac{1}{2}r|\xi_1|]\right)h\epsilon_2^{-1}
 \end{align*}
 for all $x\in\R^d$. Furthermore, for $\lambda:=h\epsilon_1^{-1}\in (0,1]$, we can use
 Lemma~\ref{lem:deriv2} to estimate
 \begin{align*}
  Au^\epsilon(t-h)-Au^\epsilon(t)
  &=\E[Du^\epsilon(t-h)\xi_1]-\E[Du^\epsilon(t)\xi_1] \\
  &\leq\lambda\E\left[\frac{(Du^\epsilon(t-h)-Du^\epsilon(t))\xi_1}{\lambda}+Du^\epsilon(t)\xi_1\right]
  	+\lambda\E[Du^\epsilon(t)\xi_1] \\
  &\leq\lambda\E\left[\frac{1}{\lambda}\int_0^h \|\partial_t Du^\epsilon(t-s)\|_\infty |\xi_1|\,\d s+Du^\epsilon(t)\xi_1\right]
  	+\lambda\E[Du^\epsilon(t)\xi_1] \\
  &\leq\lambda\E\left[\frac{h\epsilon_1^{-1}}{\lambda} d^{\frac{1}{2}}(c_r\epsilon_1+r\epsilon_2)\epsilon_2^{-1} b_{1,1} |\xi_1|
  	+d^\frac{1}{2}r|\xi_1|\right]+\lambda\E[d^\frac{1}{2}r|\xi_1|] \\
  &\leq\left(\E\big[\big(d^{\frac{1}{2}}(c_r\epsilon_1+r\epsilon_2)\epsilon_2^{-1} b_{1,1}+d^\frac{1}{2}r\big)|\xi_1|\big]
  	+\E[d^\frac{1}{2}r|\xi_1|]\right)h\epsilon_1^{-1}
 \end{align*}
 and
 \[ \partial_t u^\epsilon(t)-\frac{u^\epsilon(t)-u^\epsilon(t-h)}{h}
 	\leq\frac{1}{2}(c_r\epsilon_1+r\epsilon_2) b_{2,0} h\epsilon_1^{-2}. \]
 Hence, we can choose
 \begin{align*}
  \rho_2(\epsilon_1,\epsilon_2,h_n,r,t)
  &:=\Big(\E\Big[\tfrac{1}{2}drb_{0,1}|\xi_1|^2+d^\frac{1}{2}r|\xi_1|\Big]+\E[d^\frac{1}{2}r|\xi_1|]\Big)h_n\epsilon_2^{-1} \\
  & +\Big(\E\big[\big(d^\frac{1}{2}(c_r\epsilon_1+r\epsilon_2)\epsilon_2^{-1} b_{1,1}+d^\frac{1}{2}r\big)|\xi_1|\big]
  	+\E[d^\frac{1}{2}r|\xi_1|]\Big)h_n\epsilon_1^{-1} \\
  & +\tfrac{1}{2}(c_r\epsilon_1+r\epsilon_2)\epsilon_1^{-1} b_{2,0} h_n\epsilon_1^{-1}.
\end{align*}
 Similarly, one can show that Assumption~\ref{ass:rate}(iii) is satisfied with
 \begin{align*}
  \rho_3(\epsilon_1,\epsilon_2,h_n,r,t)
   &:=\Big(\E\Big[\tfrac{1}{2}drb_{0,1}|\xi_1|^2+d^\frac{1}{2}r|\xi_1|\Big]+\E[d^\frac{1}{2}r|\xi_1|]\Big)h_n\epsilon_2^{-1} \\
  & +\Big(\E\big[\big(d^{\frac{1}{2}}(c_r(\epsilon_1+h_n)+r\epsilon_2)\epsilon_2^{-1} b_{1,1}+d^\frac{1}{2}r\big)|\xi_1|\big]
  	+\E[d^\frac{1}{2}r|\xi_1|]\Big)h_n\epsilon_1^{-1} \\
  & +\tfrac{1}{2}(c_r(\epsilon_1+h_n)+r\epsilon_2)\epsilon_1^{-1} b_{2,0} h_n\epsilon_1^{-1}.
\end{align*}
 It follows from Theorem~\ref{thm:rate2} that
 \[ -c_r^-\sqrt{h_n}\leq\frac{1}{n}\E\left[nf\left(\frac{1}{n}\sum_{i=1}^n \xi_i\right)\right]-\E[f(\zeta)]
 	\leq c_r^+\sqrt{h_n} \]
 for all $r\geq 1$, $f\in\Lipb(r)$ and $n\in\N$, where
 \begin{align}
  c_{r,t}^- &:=8r+5\E[d^\frac{1}{2}r|\xi_1|]t
  	+\E\big[\tfrac{1}{2}drb_{0,1}|\xi_1|^2+d^\frac{1}{2}r|\xi_1|\big]t \nonumber \\
  	&\quad\; +\E\big[(d^\frac{1}{2}(c_r+r)b_{1,1}+d^\frac{1}{2}r)|\xi_1|\big]t
   +\tfrac{1}{2}(c_r+r)b_{2,0}t,  \label{eq:LLN.cr-} \\
  c_{r,t}^+ &:=8r+5\E[d^\frac{1}{2}r|\xi_1|]t
  	+\E\big[\tfrac{1}{2}drb_{0,1}|\xi_1|^2+d^\frac{1}{2}r|\xi_1|\big]t \nonumber \\
  	&\quad\; +\E\big[(d^\frac{1}{2}(2c_r+r)b_{1,1}+d^\frac{1}{2}r)|\xi_1|\big]t
  +\tfrac{1}{2}(2c_r+r)b_{2,0}t \label{eq:LLN.cr+}
 \end{align}
 and $c_r=2r+\E[d^\frac{1}{2}r|\xi_1|]$.
\end{proof}

\subsection{Central limit theorem for convex expectations}

Let $\kappa\equiv 1$ and denote by $\S^d_+$ the set of all symmetric positive-semidefinite
$d\times d$-matrices. Let $(\xi_n)_{n\in\N}\subset\H^d$ be an iid sequence of random
vectors $\xi_n\colon\Omega\to\R^d$ on a convex expectation space $(\Omega,\H,\E)$
with $\E[a\xi_1]=0$ for all $a\in\R^d$ and $\lim_{c\to\infty}\E[(|\xi_1|^2-c)^+]=0$.
In addition, we assume that $f(X)\in\H$ for all $f\in\Cb$ and $X\in\H^d$. In~\cite[Theorem 4.1]{BK22},
it was shown that
\begin{equation} \label{eq:CLT.conv}
\frac{1}{n} \E\left[nf\left(\frac{1}{\sqrt{n}}\sum_{i=1}^n \xi_i\right)\right]\to\bar{\E}[f(\zeta)]
 \quad\mbox{for all } f\in\Cb,
\end{equation}
where $\zeta\colon\bar{\Omega}\to\R^d$ is $G$-distributed with
$G(a):=\E[\frac{1}{2}\xi_1^T a\xi_1]$ for all $a\in\R^{d\times d}$. This means
that $\bar{\E}[f(\zeta)]=(S(1)f)(0)$ for all $f\in\Cb$, where $(S(t))_{t\geq 0}$ denotes
the unique strongly continuous convex monotone semigroup on $\Cb$ with generator
\[ (Af)(x)=\E\left[\tfrac{1}{2}\xi_1^T D^2f(x)\xi_1\right]
	\quad\mbox{for all } f\in\Cb^2 \mbox{ and } x\in\R^d. \]
Here, we denote by $\Cb^2$ the space of all bounded continuous twice continuously
differentiable functions $f\colon\R^d\to\R$ such that the first and second derivative
are bounded.
W.l.o.g., we subsequently assume that the random vectors $(\xi_n)_{n\in\N}$
and $\zeta$ are defined on the same convex expectation space. The converge
in equation~\eqref{eq:CLT.conv} was obtained by applying Theorem~\ref{thm:cher}
with $I(0):=\id_{\Cb}$ and
\[ (I(t)f)(x):=t\E\left[\frac{1}{t}f(x+\sqrt{t}\xi_1)\right] \]
for all $t>0$, $f\in\Cb$ and $x\in\R^d$. Under an additional moment condition,
we are now able to prove a convergence rate for equation~\eqref{eq:LLN.conv}.
If $\E$ is sublinear, condition~\eqref{eq:CLT.ass} is satisfied with $a=\E[|\xi_1|]$ and $p=1$.

\begin{theorem} \label{thm:CLT}
 Let $\E[c|\xi_1|^3]<\infty$ for all $c\geq 0$.\ Suppose that there exist $a\geq 0$ and 
 $p\geq 1$ with
 \begin{equation} \label{eq:CLT.ass}
  \E[\lambda(c_1|\xi_1|^2+c_2|\xi_1|^3)]\leq a\lambda^p\E[c_1|\xi_1|^2+c_2|\xi_1|^3]
   \quad\mbox{for all } c_1,c_2\geq 0 \mbox{ and } \lambda\geq 1.
 \end{equation}
 Then, for every $r\geq 1$, $f\in\Lipb(r)$ and $n\in\N$,
 \[ -c_{r,t}^-h_n^\gamma\leq\frac{1}{n}\E\left[n f\left(\frac{1}{\sqrt{n}}\sum_{i=1}^n \xi_i\right)\right]-\E[f(\zeta)]
 	\leq c_{r,t}^+h_n^\gamma \quad\mbox{with}\quad\gamma:=\frac{1}{4+2p}, \]
 where $c_{r,t}^\pm$ are defined by equation~\eqref{eq:CLT.cr-} and equation~\eqref{eq:CLT.cr+}.
\end{theorem}
\begin{proof}
 First, we verify Assumption~\ref{ass:rate}(i). For every $h>0$, $f\in\Cb^2$ and $x\in\R^d$,
 Taylor's formula implies
 \[ f(x+\sqrt{h}\xi_1)=f(x)+Df(x)\sqrt{h}\xi_1+h\int_0^1 \xi_1^T D^2f(x+s\sqrt{h}\xi_1)\xi_1(1-s)\,\d s \]
 Since~\cite[Lemma~B.2]{BK22} yields $\E[X+b\xi_1]=\E[X]$ for all $b\in\R^d$
 and $X\in\H^d$, we obtain
 \begin{align*}
  \left|\frac{I(h)f-f}{h}\right|(x)
  &=\left|\E\left[\frac{f(x+\sqrt{h}\xi_1)-f(x)}{h}\right]\right| \\
  &\leq\E\left[\int_0^1 |\xi_1^T D^2f(x+s\sqrt{h}\xi_1)\xi_1|(1-s)\,\d s\right]
  \leq\E\left[\frac{1}{2}\|D^2 f\|_\infty |\xi_1|^2\right].
 \end{align*}
 Hence, we can choose $\rho_1(x_1,x_2):=\E[\frac{d}{2}x_2|\xi_1|^2]\leq a\E[\frac{d}{2}|\xi_1|^2]x_2^p$ 
 for all $x_1,x_2\geq 1$ and Corollary~\ref{cor:time} can be applied with $a_1(x_1):=0$, 
 $a_2:=a\E[\frac{d}{2}|\xi_1|^2]$, $\alpha:=\frac{1}{1+p}$, $\omega:=0$ and 
 $c_r:=c_{r,T}:=(2r+a_2 b_{0,1}^pr^p)$.

 Second, we verify Assumption~\ref{ass:rate}(ii) and~(iii). Let $\epsilon=(\epsilon_1,\epsilon_2)$ 
 with $\epsilon_1,\epsilon_2>0$, $n\in\N$, $r\geq 1$, $t\geq h_n$ and $f\in\Lipb(r)$. 
 With $h:=h_n$, it holds
 \begin{align}
  \partial_t u^\epsilon(t)-Au^\epsilon(t)-\frac{u^\epsilon(t)-I(h)u^\epsilon(t-h)}{h}
  &=\frac{I(h)u^\epsilon(t-h)-u^\epsilon(t-h)}{h}-Au^\epsilon(t-h) \nonumber \\
  &\quad\; +Au^\epsilon(t-h)-Au^\epsilon(t) \nonumber \\
  &\quad\; +\partial_t u^\epsilon(t)-\frac{u^\epsilon(t)-u^\epsilon(t-h)}{h}. \label{eq:CLT.ub}
 \end{align}
 Applying Lemma~\ref{lem:lambda} with $\lambda:=\sqrt{h}\epsilon_2^{-1}\in (0,1]$,
 Lemma~\ref{lem:deriv} and condition~\eqref{eq:CLT.ass} yields
 \begin{align*}
  &\left(\frac{I(h)u^\epsilon(t-h)-u^\epsilon(t-h)}{h}-Au^\epsilon(t-h)\right)(x) \\
  &=\E\left[\int_0^1 \xi_1^T D^2u^\epsilon(t-h,x+s\sqrt{h}\xi_1)\xi_1 (1-s)\,\d s\right]
  	-\E\left[\frac{1}{2}\xi_1^T D^2u^\epsilon(t-h,x)\xi_1\right] \\
  &\leq\lambda\E\left[\frac{1}{\lambda}\int_0^1
  	\xi_1^T(D^2u^\epsilon(t-h,x+s\sqrt{h}\xi_1)-D^2u^\epsilon(t-h,x))\xi_1(1-s)\,\d s
  	+\frac{1}{2}\xi_1^T D^2u^\epsilon(t-h,x)\xi_1\right] \\
  &\quad\; -\lambda\E\left[\frac{1}{2}\xi_1^T D^2u^\epsilon(t-h,x)\xi_1\right] \\
  &\leq\lambda\E\left[\frac{1}{\lambda}\int_0^1 \|D^3u^\epsilon(t-h)\|_\infty |\xi_1|^3\sqrt{h}s(1-s)\,\d s
  	+\frac{1}{2}\|D^2u^\epsilon(t-h)\|_\infty |\xi_1|^2\right] \\
  &\quad\; +\lambda\E\left[\frac{1}{2}\|D^2u^\epsilon(t-h)\|_\infty |\xi_1|^2\right] \\
  &\leq\lambda\E\left[\left(\frac{\sqrt{h}\epsilon_2^{-1}}{6\lambda}d^\frac{3}{2}rb_{0,2}|\xi_1|^3
  	+\frac{1}{2}d rb_{0,1}|\xi_1|^2\right)\epsilon_2^{-1}\right]
  	+\lambda\E\left[\frac{1}{2}rb_{0,1}|\xi_1|^2\epsilon_2^{-1}\right] \\
  &\leq a\left(\E\left[\frac{1}{6}d^\frac{3}{2}rb_{0,2}|\xi_1|^3+\frac{1}{2} d rb_{0,1}|\xi_1|^2\right]
  	+\E\left[\frac{1}{2}drb_{0,1}|\xi_1|^2\right]\right)\sqrt{h}\epsilon_2^{-(1+p)}
 \end{align*}
 for all $x\in\R^d$. In addition, for the choice $\lambda:=h\epsilon_1^{-1}$, Lemma~\ref{lem:deriv},
 Lemma~\ref{lem:deriv2} and condition~\eqref{eq:CLT.ass} imply
 \begin{align*}
  &Au^\epsilon(t-h)-Au^\epsilon(t) \\
  &=\E\left[\frac{1}{2}\xi_1^T D^2u^\epsilon(t-h)\xi_1\right]-\E\left[\frac{1}{2}\xi_1^T D^2u^\epsilon(t)\xi_1\right] \\
  &\leq\lambda\E\left[\frac{1}{2\lambda}\xi_1^T(D^2u^\epsilon(t-h)-D^2u^\epsilon(t))\xi_1
  	+\frac{1}{2}\xi_1^T D^2u^\epsilon(t)\xi_1\right]
  	-\lambda\E\left[\frac{1}{2}\xi_1^T D^2u^\epsilon(t)\xi_1\right] \\
  &\leq\lambda\E\left[\frac{1}{2\lambda}\int_0^h \|\partial_t D^2u^\epsilon(t-s)\|_\infty |\xi_1|^2\,\d s
  	+\frac{1}{2}\|D^2u^\epsilon(t)\|_\infty |\xi_1|^2\right]
  	-\lambda\E\left[\frac{1}{2}\|D^2u^\epsilon(t)\|_\infty |\xi_1|^2\right] \\
  &\leq\lambda\E\left[\left(\frac{h\epsilon_1^{-1}}{2\lambda}
  	d\big(c_r\epsilon_1^\frac{1}{1+p}+r\epsilon_2\big)\epsilon_2^{-1} b_{1,2}
  	+\frac{1}{2}drb_{0,1}\right)|\xi_1|^2\epsilon_2^{-1}\right]
  	+\lambda\E\left[\frac{1}{2}d rb_{0,1}|\xi_1|^2\epsilon_2^{-1}\right] \\
  &\leq a\left(\E\left[\left(\frac{d}{2}\big(c_r\epsilon_1^\frac{1}{1+p}+r\epsilon_2\big)\epsilon_2^{-1} b_{1,2}
  	+\frac{1}{2} drb_{0,1}\right)|\xi_1|^2\right]+\E\left[\frac{1}{2}drb_{0,1}|\xi_1|^2\right]\right)h\epsilon_1^{-1}\epsilon_2^{-p}
 \end{align*}
 and
 \[ \partial_t u^\epsilon(t)-\frac{u^\epsilon(t)-u^\epsilon(t-h)}{h}
 	\leq\frac{1}{2}(c_r\epsilon_1^\frac{1}{1+p}+r\epsilon_2)\epsilon_1^{-\frac{1}{1+p}}
 	b_{2,0} h\epsilon_1^{-\frac{1+2p}{1+p}}. \]
 Hence, we can choose
 \begin{align*}
  &\rho_2(\epsilon_1,\epsilon_2,h_n,r,t)
  :=a\Big(\E\big[\tfrac{1}{6}d^\frac{3}{2}rb_{0,2}|\xi_1|^3+\frac{d}{2}rb_{0,1}|\xi_1|^2\big]
  +\E\big[\tfrac{d}{2}rb_{0,1}|\xi_1|^2\big]\Big)\sqrt{h_n}\epsilon_2^{-(1+p)} \\
  &\quad\;+a\Big(\E\Big[\Big(\tfrac{d}{2}\big(c_r\epsilon_1^{\frac{1}{1+p}}+r\epsilon_2\big)\epsilon_2^{-1} b_{1,2}
  +\tfrac{d}{2} rb_{0,1}\Big)|\xi_1|^2\Big]+\E\big[\tfrac{d}{2}rb_{0,1}|\xi_1|^2\big]\Big)h_n\epsilon_1^{-1}\epsilon_2^{-p} \\
  &\quad\;+\tfrac{1}{2}\big(c_r\epsilon_1^{\frac{1}{1+p}}+r\epsilon_2\big)\epsilon_1^{-{\frac{1}{1+p}}}
  b_{2,0} h_n\epsilon_1^{-\frac{1+2p}{1+p}}.
 \end{align*}
 Similarly, one can show that Assumption~\ref{ass:rate}(iii) is satisfied with
 \begin{align*}
  &\rho_3(\epsilon_1,\epsilon_2,h_n,r,t)
  :=a\Big(\E\big[\tfrac{1}{6}d^\frac{3}{2}rb_{0,2}|\xi_1|^3+\tfrac{d}{2}rb_{0,1}|\xi_1|^2\big]
  +\E\big[\tfrac{1}{2}drb_{0,1}|\xi_1|^2\big]\Big)\sqrt{h_n}\epsilon_2^{-(1+p)} \\
  &+a\Big(\E\Big[\Big(\tfrac{d}{2}\big(c_r(\epsilon_1+h_n)^{\frac{1}{1+p}}+r\epsilon_2\big)\epsilon_2^{-1} b_{1,2}
  +\tfrac{d}{2} rb_{0,1}\Big)|\xi_1|^2\Big]+\E\big[\tfrac{d}{2}rb_{0,1}|\xi_1|^2\big]\Big)h_n\epsilon_1^{-1}\epsilon_2^{-p} \\
  &+\tfrac{1}{2}\big(c_r (\epsilon_1+h_n)^{\frac{1}{1+p}}+r\epsilon_2\big)\epsilon_1^{-{\frac{1}{1+p}}}
  b_{2,0} h_n\epsilon_1^{-\frac{1+2p}{1+p}}.\\
  \end{align*}
 It follows from Theorem~\ref{thm:rate2} that
 \[ -c_r^-h_n^\gamma\leq\E\left[f\left(\frac{1}{\sqrt{n}}\sum_{i=1}^n \xi_i\right)\right]-\E[f(\zeta)]
 	\leq c_r^+h_n^\gamma \quad\mbox{with}\quad\gamma:=\frac{1}{4+2p} \]
 for all $r\geq 1$, $f\in\Lipb(r)$ and $n\in\N$, where 
 \begin{align}
  c_{r,t}^- &:=8r+3a\E\big[\tfrac{d}{2}|\xi_1|^2\big]b_{0,1}^p r^p+2a\E\big[\tfrac{d}{2}rb_{0,1}|\xi_1|^2\big]t 
  	+a\E\big[\tfrac{1}{6}d^\frac{3}{2}rb_{0,2}|\xi_1|^3+\tfrac{d}{2}rb_{0,1}|\xi_1|^2\big]t  \nonumber \\
  	&\quad\; +a\E\big[\tfrac{d}{2}\big((c_r+r)b_{1,2}+rb_{0,1}\big)|\xi_1|^2\big]+\tfrac{1}{2}(c_r+r)b_{2,0} 
  	\label{eq:CLT.cr-}  \\
  c_{r,t}^+ &:=8r+3a\E\big[\tfrac{d}{2}|\xi_1|^2\big]b_{0,1}^p r^p+2a\E\big[\tfrac{d}{2}rb_{0,1}|\xi_1|^2\big]t 
  	+a\E\big[\tfrac{1}{6}d^\frac{3}{2}rb_{0,2}|\xi_1|^3+\tfrac{d}{2}rb_{0,1}|\xi_1|^2\big]t  \nonumber \\
  	&\quad\; +a\E\big[\tfrac{d}{2}\big((2c_r+r)b_{1,2}+rb_{0,1}\big)|\xi_1|^2\big]+\tfrac{1}{2}(2c_r+r)b_{2,0} 
  \label{eq:CLT.cr+}
 \end{align}
 and $c_r=2r+a\E[\frac{d}{2}|\xi_1|^2]b_{0,1}^p r^p$.
\end{proof}

In the one-dimensional case, the convergence rate can be improved by imposing
an additional moment condition which allows for a higher order Taylor expansion
in the computation of the consistency error.

\begin{theorem} \label{thm:CLT2}
 Let $d=1$ and suppose that there exist $a\geq 0$ and $p\geq 1$ with
 \begin{equation} \label{eq:CLT.ass2}
  \E[\lambda(c_1|\xi_1|^2+c_2|\xi_1|^3)]\leq a\lambda^p\E[c_1|\xi_1|^2+c_2|\xi_2|^3]
   \quad\mbox{for all } c_1,c_2\geq 0 \mbox{ and } \lambda\geq 1.
 \end{equation}
 Furthermore, let $\E[c\xi_1^3]=0$ for all $c\in\R$ and $\E[c\xi_1^4]<\infty$ for all $c\geq 0$.
 Then, for every $r\geq 1$, $f\in\Lipb(r)$ and $n\in\N$,
 \[ -c_r^-h_n^\gamma\leq\frac{1}{n}\E\left[nf\left(\frac{1}{\sqrt{n}}\sum_{i=1}^n \xi_i\right)\right]-\E[f(\zeta)]
 	\leq c_r^+h_n^\gamma \quad\mbox{with}\quad\gamma:=\frac{1}{2+2p}, \]
 where $c_r^\pm$ are defined by equation~\eqref{eq:CLT2.cr-} and equation~\eqref{eq:CLT2.cr+}.
\end{theorem}
\begin{proof}
 By choosing $\rho_1(x_1,x_2):=\E[\frac{1}{2}x_2|\xi_1|^2]\leq a\E[\frac{1}{2}|\xi_1|^2] x_2^p$,
 Assumption~\ref{ass:rate}(i) is satisfied and Corollary~\ref{cor:time} can be applied with 
 $a_1(x_1):=0$, $a_2:=a\E[\frac{1}{2}|\xi_1|^2]$, $\alpha:=\frac{1}{1+p}$, $\omega:=0$ and 
 $c_r:=c_{r,T}:=(2r+a_2 b_{0,1}^pr^p)$.
 Let $\epsilon=(\epsilon_1,\epsilon_2)$ with $\epsilon_1,\epsilon_2>0$,
 $n\in\N$, $r\geq 1$, $t\geq h_n$ and $f\in\Lipb(r)$. With $h:=h_n$, Taylor's formula implies
 \begin{align*}
  &u^\epsilon(t-h,x+\sqrt{h}\xi_1)
  =u^\epsilon(t-h,x)+Du^\epsilon(t-h,x)\sqrt{h}\xi_1+\frac{1}{2}D^2u^\epsilon(t-h,x)h\xi_1^2 \\
  &\quad +\frac{1}{6}D^3u^\epsilon(t-h,x)h^\frac{3}{2}\xi_1^3
  	+\frac{1}{6}\int_0^1 D^4u^\epsilon(t-h,x+s\sqrt{h}\xi_1)h^2\xi_1^4(1-s)^3\,\d s.
 \end{align*}
 Since~\cite[Lemma~B.2]{BK22} yields $\E[X+b\xi_1]=\E[X+b\xi_1^3]=\E[X]$ for all $b\in\R^d$
 and $X\in\H^d$, we can apply Lemma~\ref{lem:lambda} with $\lambda:=h\epsilon_2^{-2}\in (0,1]$,
 Lemma~\ref{lem:deriv} and condition~\eqref{eq:CLT.ass2} to obtain
 \begin{align*}
  &\frac{I(h)u^\epsilon(t-h)-u^\epsilon(t-h)}{h}-Au^\epsilon(t-h) \\
  &=\E\left[\frac{1}{2}D^2u^\epsilon(t-h,x)\xi_1^2
  	+\frac{1}{6}\int_0^1 D^4u^\epsilon(t-h,x+s\sqrt{h}\xi_1)h\xi_1^4(1-s)^3\,\d s\right] \\
  &\quad\; -\E\left[\frac{1}{2}D^2u^\epsilon(t-h,x)\xi_1^2\right] \\
  &\leq\lambda\E\left[\frac{1}{6\lambda}\int_0^1 \|D^4u^\epsilon(t-h)\|_\infty h\xi_1^4(1-s)^3\,\d s
  	+\frac{1}{2}\|D^2u^\epsilon(t-h)\|_\infty\xi_1^2\right] \\
  &\quad\; -\lambda\E\left[\frac{1}{2}\|D^2u^\epsilon(t-h)\|_\infty\xi_1^2\right] \\
  &\leq\lambda\E\left[\left(\frac{h\epsilon_2^{-2}}{24\lambda}rb_{0,3}\xi_1^4+\frac{1}{2}rb_{0,1}\xi_1^2\right)\epsilon_2^{-1}\right]
  	+\lambda\E\left[\frac{1}{2}rb_{0,1}\xi_1^2\epsilon_2^{-1}\right] \\
  &\leq a\left(\E\left[\frac{1}{24}rb_{0,3}\xi_1^4+\frac{1}{2}rb_{0,1}\xi_1^2\right)
  	+\E\left[\frac{1}{2}rb_{0,1}\xi_1^2\right]\right)h\epsilon_2^{-(2+p)}.
 \end{align*}
 Combining the previous estimate with equation~\eqref{eq:CLT.ub} and the
 estimates from Theorem~\ref{thm:CLT} shows that Assumption~\ref{ass:rate}(ii)
 is satisfied with
 \begin{align*}
  &\rho_2(\epsilon_1,\epsilon_2,h_n,r,t)
  :=a\Big(\E\big[\tfrac{1}{24}rb_{0,3}\xi_1^4+\tfrac{1}{2}rb_{0,1}\xi_1^2\big]
  	+\E\big[\tfrac{1}{2}rb_{0,1}\xi_1^2\big]\Big)h_n\epsilon_2^{-(2+p)} \\
  &\quad\; +\Big(\E\Big[\tfrac{1}{2}\Big(\big(c_r\epsilon_1^\frac{1}{1+p}+r\epsilon_2\big)\epsilon_2^{-1}b_{1,2}
  	+rb_{0,1}\Big)|\xi_1|^2\Big]+\E\big[\tfrac{1}{2}rb_{0,1}|\xi_1|^2\big]\Big)h_n\epsilon_1^{-1}\epsilon_2^{-p} \\
  &\quad\; +\tfrac{1}{2}\big(c_r\epsilon_1^\frac{1}{1+p}+r\epsilon_2\big)\epsilon_1^{-\frac{1}{1+p}}
 	b_{2,0} h_n\epsilon_1^{-\frac{1+2p}{1+p}}.
 \end{align*}
  Similarly, one can show that Assumption~\ref{ass:rate}(iii) is satisfied with
 \begin{align*}
  &\rho_3(\epsilon_1,\epsilon_2,h_n,r,t)
  :=a\Big(\E\big[\tfrac{1}{24}rb_{0,3}|\xi_1|^4+\tfrac{1}{2}rb_{0,1}|\xi_1|^2\big]
  	+\E\big[\tfrac{1}{2}rb_{0,1}|\xi_1|^2\big]\Big)h_n\epsilon_2^{-(2+p)} \\
  &\quad\; +\Big(\E\Big[\tfrac{1}{2}\Big(\big(c_r(\epsilon_1+h_n)^\frac{1}{1+p}+r\epsilon_2\big)\epsilon_2^{-1} b_{1,2}
  	+rb_{0,1}\Big)|\xi_1|^2\Big]+\E\big[\tfrac{1}{2}rb_{0,1}|\xi_1|^2\big]\Big)h_n\epsilon_1^{-1}\epsilon_2^{-p} \\
  &\quad\; +\tfrac{1}{2}\big(c_r(\epsilon_1+h_n)^\frac{1}{1+p}+r\epsilon_2\big)\epsilon_1^{-\frac{1}{1+p}}
 	b_{2,0} h_n\epsilon_1^{-\frac{1+2p}{1+p}}.
 \end{align*}
 It follows from Theorem~\ref{thm:rate2} that
 \[ -c_r^-h_n^\gamma\leq\frac{1}{n}\E\left[nf\left(\frac{1}{\sqrt{n}}\sum_{i=1}^n \xi_i\right)\right]-\E[f(\zeta)]
 	\leq c_r^+h_n^\gamma \quad\mbox{with}\quad\gamma:=\frac{1}{2+2p} \]
 for all $r\geq 1$, $f\in\Lipb(r)$ and $n\in\N$, where
 \begin{align}
  c_{r,t}^- &:=8r+3a\E\big[\tfrac{1}{2}|\xi_1|^2\big]b_{0,1}^p r^p+2a\E\big[\tfrac{1}{2}rb_{0,1}|\xi_1|^2\big]t 
  	+a\E\big[\tfrac{1}{24}rb_{0,3}|\xi_1|^4+\tfrac{1}{2}rb_{0,1}|\xi_1|^2\big] \nonumber \\
  	&\quad\; +\E\big[\tfrac{1}{2}\big((c_r+r)b_{1,2}+rb_{0,1}\big)|\xi_1|^2\big]
  	+\tfrac{1}{2}(c_r+r)b_{2,0},  \label{eq:CLT2.cr-} \\
  c_{r,t}^+ &:=8r+3a\E\big[\tfrac{1}{2}|\xi_1|^2\big]b_{0,1}^p r^p+2a\E\big[\tfrac{1}{2}rb_{0,1}|\xi_1|^2\big]t 
  	+a\E\big[\tfrac{1}{24}rb_{0,3}|\xi_1|^4+\tfrac{1}{2}rb_{0,1}|\xi_1|^2\big] \nonumber \\
  	&\quad\; +\E\big[\tfrac{1}{2}\big((2c_r+r)b_{1,2}+rb_{0,1}\big)|\xi_1|^2\big]
  	+\tfrac{1}{2}(2c_r+r)b_{2,0} \label{eq:CLT2.cr+}
 \end{align}
 and $c_r=2r+a\E[\frac{d}{2}|\xi_1|^2]b_{0,1}^p r^p$.
\end{proof}

\appendix

\section{Basic convexity estimates}

\begin{lemma} \label{lem:lambda}
 Let $X$ be a vector space and $\Phi\colon X\to\R$ be a convex functional. Then,
 \[ \Phi(x)-\Phi(y)\leq\lambda\left(\Phi\left(\frac{x-y}{\lambda}+y\right)-\Phi(y)\right)
 	\quad\mbox{for all } x,y\in X \mbox{ and } \lambda\in (0,1]. \]
\end{lemma}
\begin{proof}
 For every $x,y\in X$ and $\lambda\in (0,1]$, 
 \begin{align*}
  \Phi(x)-\Phi(y)
  &=\Phi\left(\lambda\left(\frac{x-y}{\lambda}+y\right)+(1-\lambda)y\right)-\Phi(y) \\
  &\leq\lambda\Phi\left(\frac{x-y}{\lambda}+y\right)+(1-\lambda)\Phi(y)-\Phi(y) \\
  &=\lambda\left(\Phi\left(\frac{x-y}{\lambda}+y\right)-\Phi(y)\right). \qedhere
 \end{align*}
\end{proof}

\begin{corollary} \label{cor:kappa}
 Let $\Phi$ be a convex operator $\Phi\colon\Ck\to\Ck$ such that there exists $c\geq 0$ with
 $\|\Phi(f)\|_\kappa\leq c\|f\|_\kappa$ for all $f\in\Ck$. Then,
 \[ \Phi\left(f+\frac{a}{\kappa}\right)\leq\Phi(f)+\frac{c|a|}{\kappa}
 	\quad\mbox{for all } f\in\Ck \mbox{ and } a\in\R. \]
\end{corollary}
\begin{proof}
 Let $f\in\Ck$ and $a\in\R$. For every $\lambda\in (0,1)$,
 \begin{align*}
  \Phi\left(f+\frac{a}{\kappa}\right)
  &\leq\lambda\Phi\left(\frac{1}{\lambda}f\right)+(1-\lambda)\Phi\left(\frac{a}{(1-\lambda)\kappa}\right) \\
  &\leq\lambda\Phi\left(\frac{1}{\lambda}f\right)+(1-\lambda)\frac{c|a|}{(1-\lambda)\kappa}
  =\lambda\Phi\left(\frac{1}{\lambda}f\right)+\frac{c|a|}{\kappa}.
 \end{align*}
 Furthermore, Lemma~\ref{lem:lambda} implies $\lambda\Phi\left(\frac{1}{\lambda}f\right)\to\Phi(f)$
 as $\lambda\to 1$.
\end{proof}

\begin{lemma} \label{lem:Jensen}
 Let $\Phi\colon\Ck\to\R$ be a convex monotone functional that is continuous from above.
 Furthermore, let $\mu$ be a probability measure on $\B(\R_+\times\R^d)$ and
 $u\colon\R_+\times\R^d\to\R$ be a function such that $u(t, \cdot)\in\Ck$ and
 \[ \int_{\R_+\times\R^d}u(s+t, \cdot+y)\,\mu(\d s\times\d y)\in\Ck
 	\quad\mbox{for all } t\geq 0. \]
 Then, for every $t\geq 0$,
 \[ \Phi\left(\int_{\R_+\times\R^d}u(s+t, \cdot+y)\,\mu(\d s\times\d y)\right)
 	\leq\int_{\R_+\times\R^d}\Phi(u(s+t, \cdot+y))\,\mu(\d s\times\d y). \]
\end{lemma}
\begin{proof}
 It follows from~\cite[Theorem~2.2]{BCK19} that
 \[ \Phi(f)=\sup_{\nu\in\ca^+_\kappa}\left(\int_{\R^d}f\,\d\nu-\Phi^*(\nu)\right)
 	\quad\mbox{for all } f\in\Ck, \]
 where $\ca^+_\kappa$ consists of all Borel-measures $\nu\colon\B(\R^d)\to [0,\infty]$
 with $\int_{\R^d}\frac{1}{\kappa}\,\d\nu<\infty$ and
 \[ \Phi^*\colon\ca^+_\kappa\to [0,\infty], \; \nu\mapsto\sup_{f\in\Ck}
 	\left(\int_{\R^d}f\,\d\nu-\Phi(g)\right) \]
 denotes the convex conjugate. We use Fubini's theorem to obtain
 \begin{align*}
  &\Phi\left(\int_{\R_+\times\R^d}u(s+t, \cdot+y)\,\mu(\d s\times\d y)\right) \\
  &=\sup_{\nu\in\ca^+_\kappa}\left(\int_{\R^d}\int_{\R_+\times\R^d}
  	u(s+t, y+z)\,\mu(\d s\times\d y)\,\nu(\d z)-\Phi^*(\nu)\right) \\
  &=\sup_{\nu\in\ca^+_\kappa}\int_{\R_+\times\R^d}\left(\int_{\R^d}
  	u(s+t, y+z)\,\nu(\d z)-\Phi^*(\nu)\right)\mu(\d s\times\d y) \\
  &\leq\int_{\R_+\times\R^d}\Phi(u(s+t, \cdot+y)\,\mu(\d s\times\d y). \qedhere
 \end{align*}
\end{proof}

\bibliographystyle{abbrv}
\bibliography{ConvRate.bib}

\end{document}